\documentclass[letterpaper,english]{svjour3}
\usepackage[T1]{fontenc}
\usepackage[latin9]{inputenc}
\usepackage{float}
\usepackage{units}
\usepackage{url}
\usepackage{amsmath}
\usepackage{amssymb}
\usepackage{graphicx}

\makeatletter

\pdfpageheight\paperheight
\pdfpagewidth\paperwidth

\providecommand{\tabularnewline}{\\}
\floatstyle{ruled}
\newfloat{algorithm}{tbp}{loa}
\providecommand{\algorithmname}{Algorithm}
\floatname{algorithm}{\protect\algorithmname}

\usepackage{amsmath,graphicx}

\usepackage{algorithmic} 
\usepackage{algorithm} 

\DeclareMathOperator*{\argmin}{arg\,min}
\DeclareMathOperator*{\prox}{prox}

\usepackage[numbers,sort,compress]{natbib}

\makeatother

\usepackage{babel}

\begin{document}

\title{Fast Proximal Gradient Methods for Nonsmooth Convex Optimization
for Tomographic Image Reconstruction}

\author{Elias S. Helou, Marcelo V. W. Zibetti, and Gabor T. Herman}

\institute{Elias S. Helou, Instituto de Ciências Matemáticas e de Computação,
elias@icmc.usp.br\\
Marcelo V. W. Zibetti, Center for Advanced Imaging Innovation and
Research (CAI\textsuperscript{2}R), New York University School of
Medicine, Marcelo.WustZibetti@nyumc.org\\
Gabor T. Herman, PhD Program in Computer Science, City University
of New York, gabortherman@yahoo.com}
\maketitle
\begin{abstract}
The Fast Proximal Gradient Method (FPGM) and the Monotone FPGM (MFPGM)
for minimization of nonsmooth convex functions are introduced and
applied to tomographic image reconstruction. Convergence properties
of the sequence of objective function values are derived, including
a $O\left(1/k^{2}\right)$ non-asymptotic bound. The presented theory
broadens current knowledge and explains the convergence behavior of
certain methods that are known to present good practical performance.
Numerical experimentation involving computerized tomography image
reconstruction shows the methods to be competitive in practical scenarios.
Experimental comparison with Algebraic Reconstruction Techniques are
performed uncovering certain behaviors of accelerated Proximal Gradient
algorithms that apparently have not yet been noticed when these are
applied to tomographic image reconstruction.
\textbf{Mathematics Subject Classificatio}n$\;$65K05, 68U10, 90C25,
90C06, 92C55, 94A08
\end{abstract}

\section{Introduction}

In this paper we introduce the Fast Proximal Gradient Methods (FPGM)
for the solution of convex minimization problems of the form
\begin{equation}
\min\quad\Psi(\mathbf{x}):=f(\mathbf{x})+\phi(\mathbf{x}),\label{eq:problem}
\end{equation}
where $f:\mathbb{R}^{n}\to\mathbb{R}$ is a smooth convex function
with Lipschitz continuous gradient and $\phi:\mathbb{R}^{n}\to\mathbb{R}$
is a proper convex function. It is not assumed that $\phi$ is smooth.
From now on, the symbols $f$ and $\phi$ will always represent functions
satisfying these assumptions.

Many practical problems fit this general model. In particular, for
image reconstruction from tomographic data, $f(\boldsymbol{x})$ is
usually going to be a measure of consistency, according to the data
acquisition model, of image $\boldsymbol{x}$ with the data coming
from the scanner whereas $\phi(\boldsymbol{x})$ is a measure of structure.
For example, $\phi$ could enforce sparsity, smoothness, or some other
desirable \emph{a priori} information known about the image being
reconstructed.

The algorithms we propose here lie between the Optimized Iterative
Shrinkage-Thresholding Algorithm (OISTA) presented and experimented
with in~\cite{Kim2016c}, but for which no convergence proof exists,
and the Overrelaxed Monotone Fast Iterative Shrinkage-Thresholding
Algorithm (OMFISTA)~\cite{Yamagishi2011}, for which convergence
proofs do exist. One contribution of the present paper is a better
understanding of the convergence properties of these algorithms. We
get there by increasing the range of allowed parameters values for
OMFISTA, consequently providing a way of justifying convergence of
OISTA. We exhibit numerical experimentation using FPGM for computerized
tomography image reconstruction \cite{Herman2009} with both real
and synthetic data and compare it with other methods from the viewpoint
of convergence speed and image quality.

The paper is organized as follows. The remaining of the present section
is dedicated to introduce the mathematical elements that will be used
in the construction of the algorithm. Section~\ref{sec:Literature-Review}
contains a non-extensive literature review about the subject of fast
proximal gradient methods and fast first-order methods. The new algorithm
is introduced in Section~\ref{sec:The-New-Method} and its convergence
is analyzed in Section~\ref{sec:Convergence-Analysis}. Numerical
experiments are presented in Section~\ref{sec:Experiments}. Finally,
Section~\ref{sec:Conclusions-and-Future} brings our conclusions
and delineates plans for future work.

\subsection{Proximal Operators}

Proximal operators are useful tools for convex optimization, presenting
an interesting balance between abstract power and practical applicability.
Given a proper convex function $\phi:\mathbb{R}^{n}\to\mathbb{R}$
and a positive real number $L$, the $L$-proximal point of $\phi$
from a point $\mathbf{x}\in\mathbb{R}^{n}$ is defined as: 
\begin{equation}
\prox_{\phi,L}(\mathbf{x}):=\argmin_{\mathbf{y}\in\mathbb{R}^{n}}\left\{ \phi(\mathbf{y})+\frac{L}{2}\left\Vert \mathbf{y}-\mathbf{x}\right\Vert ^{2}\right\} .
\end{equation}
We will, when there is no risk of confusion, refer to the $L$-proximal
point of $\phi$ from $\mathbf{x}\in\mathbb{R}^{n}$ simply as the
proximal point.

Although the computation of the proximal point involves a minimization,
in many cases of interest this computation can be performed in a finite
number of simple steps. Examples of functions $\phi$ that allow efficient
closed-form representations of $\prox_{\phi,L}$ include the $\ell_{1}$
norm of an orthonormal transformation of a vector, that is, $\phi(\mathbf{x})=\left\Vert Q\mathbf{x}\right\Vert _{1}$
where $Q\in\mathbb{C}^{n\times n}$ with $Q^{*}Q=I$, and the nuclear
norm of a matrix $M\in\mathbb{C}^{m\times n}$, i.e., $\phi(M)=\left\Vert M\right\Vert _{*}=\sum_{i=1}^{\min\{m,n\}}|\sigma_{i}|$,
where the $\sigma_{i}$ are the singular values of $M$. For some
other important functions, such as the Total Variation ($TV$), there
are effective iterative procedures than can quickly approximate $\prox_{\phi,L}(\mathbf{x})$~\cite{Beck2009a}.

\subsection{Proximal Gradient Methods}

Let us define the proximal gradient operator as:
\begin{equation}
P_{L,f,\phi}(\mathbf{x}):=\prox_{\phi,L}\left(\mathbf{x}-\frac{1}{L}\nabla f(\mathbf{x})\right).\label{eq:prox_grad}
\end{equation}
For simplicity, we will omit $f$ and $\phi$ from the notation unless
the omission causes ambiguity. A Proximal Gradient Method (PGM) for
the minimization of~\eqref{eq:problem} is then given by an iterative
process in which
\begin{equation}
\mathbf{x}_{k+1}=P_{L}\left(\mathbf{x}_{k}\right),
\end{equation}
see \cite{Parikh2013}. For large enough $L$, these iterations converge,
but they do so slowly, and rates of the kind $\Psi(\mathbf{x}_{k})-\Psi(\mathbf{x}^{*})=O\left(1/k\right)$,
where $\mathbf{x}^{*}\in\mathbb{R}^{n}$ is a minimizer of $\Psi$,
are often predicted theoretically and observed in practice.

For smooth problems of the form
\begin{equation}
\min\quad f(\mathbf{x}),\label{eq:problem_smooth}
\end{equation}
 Nesterov~\cite{nesterov1983method} introduced the Fast Gradient
Method (FGM)
\begin{equation}
\mathbf{x}_{k}=\mathbf{y}_{k}-\frac{1}{L}\nabla f\left(\mathbf{y}_{k}\right),\label{eq:gradient}
\end{equation}
where $\mathbf{y}_{1}=\mathbf{x}_{0}$ and
\begin{equation}
\mathbf{y}_{k+1}=\mathbf{x}_{k}+\frac{t_{k}-1}{t_{k+1}}\left(\mathbf{x}_{k}-\mathbf{x}_{k-1}\right),
\end{equation}
with $t_{1}=1$ and $t_{k}=\left(1+\sqrt{1+4t_{k-1}^{2}}\right)/2$.
Later, Beck and Teboulle proved that the same algorithmic form leads
to a fast proximal gradient algorithm, called Fast Iterative Shrinkage-Thresholding
Algorithm (FISTA)~\cite{Beck2009,Beck2009a}, when~\eqref{eq:gradient}
is replaced by
\begin{equation}
\mathbf{x}_{k}=P_{L}\left(\mathbf{y}_{k}\right).
\end{equation}
These fast algorithms have much better theoretical and practical convergence
properties than a PGM, at only a small extra computational cost, specially
in high-dimensional problems such as occur in computerized tomography
image reconstruction \cite{Herman2009}.

\section{Literature Review\label{sec:Literature-Review}}

The introduction of FISTA renewed interest in Nesterov's ideas because
the separable smooth plus nonsmooth model has a wider application
range than smooth minimization, including many image processing and
reconstruction problems within the Compressive Sensing (CS) framework~\cite{don06}.
Nesterov introduced his seminal ideas by considering First-Order Methods
(FOMs) for the solution of~\eqref{eq:problem_smooth}. A FOM is a
method with an iterative step of the form:
\begin{equation}
\mathbf{x}_{k+1}=\mathbf{x}_{k}-\frac{1}{L}\sum_{i=0}^{k}h_{k+1,i}\nabla f\left(\mathbf{x}_{i}\right),
\end{equation}
where $h_{k+1,i}\in\mathbb{R}$ for $k\in\mathbb{N}$ and $i\in\{0,1,\dots,k\}$.
For this class of methods, there is a smooth convex function $f$
with Lipschitz continuous gradient such that the sequence of iterates
of every FOM would satisfy~\cite{nes04} 
\begin{equation}
f(\mathbf{x}_{k})-f(\mathbf{x}^{*})\geq\frac{3L\left\Vert \mathbf{x}_{0}-\mathbf{x}^{*}\right\Vert ^{2}}{32(k+1)^{2}},\label{eq:lower_FOM_bound}
\end{equation}
thereby establishing a lower bound for the worst case scenario. On
the other hand, the optimality gap, that is, the difference between
the objective function value at the current iteration and the optimal
objective function value, of the FGM was shown to be bounded by
\begin{equation}
f(\mathbf{x}_{k})-f(\mathbf{x}^{*})\leq\frac{2L\left\Vert \mathbf{x}_{0}-\mathbf{x}^{*}\right\Vert ^{2}}{(k+1)^{2}}.\label{eq:upper_FISTA_bound}
\end{equation}
Since FGM is a FOM, this algorithm has a convergence rate in terms
of reduction of the objective function value with optimal leading
exponent. However, the difference between the coefficients in the
bounds~\eqref{eq:lower_FOM_bound}~and~\eqref{eq:upper_FISTA_bound}
left hope for improvement of the convergence rate of FOMs.

More recently, Kim and Fessler~\cite{Kim2016c} refined an approach
by Drori and Teboulle~\cite{drt14} and obtained a method called
Optimized Gradient Method (OGM) with the proven property that
\begin{equation}
f(\mathbf{x}_{k})-f(\mathbf{x}^{*})\leq\frac{L\left\Vert \mathbf{x}_{0}-\mathbf{x}^{*}\right\Vert ^{2}}{(k+1)^{2}},
\end{equation}
with iterations given by
\begin{equation}
\mathbf{x}_{k}=\mathbf{y}_{k}-\frac{1}{L}\nabla f\left(\mathbf{y}_{k}\right),
\end{equation}
\begin{equation}
\mathbf{y}_{k+1}=\mathbf{x}_{k}+\frac{t_{k}-1}{t_{k+1}}\left(\mathbf{x}_{k}-\mathbf{x}_{k-1}\right)+\frac{t_{k}}{t_{k+1}}\left(\mathbf{x}_{k}-\mathbf{y}_{k}\right).\label{eq:y-update-OGM}
\end{equation}

The generalization of FGM to the nonsmooth case came when Beck and
Teboulle~\cite{Beck2009} showed that the gradient step can be replaced
by a proximal gradient step leading to the FISTA
\begin{equation}
\mathbf{x}_{k}=P_{L}\left(\mathbf{y}_{k}\right),
\end{equation}
\begin{equation}
\mathbf{y}_{k+1}=\mathbf{x}_{k}+\frac{t_{k}-1}{t_{k+1}}\left(\mathbf{x}_{k}-\mathbf{x}_{k-1}\right).
\end{equation}
FISTA, similarly to FGM, satisfies
\begin{equation}
\Psi(\mathbf{x}_{k})-\Psi(\mathbf{x}^{*})\leq\frac{2L\left\Vert \mathbf{x}_{0}-\mathbf{x}^{*}\right\Vert ^{2}}{(k+1)^{2}}.
\end{equation}
A monotone version of FISTA, called MFISTA was also developed~\cite{Beck2009a},
which shares with FISTA the same bound on the optimality gap. Kim
and Fessler successfully experimented with a version of the OGM where
the gradient step is replaced by the proximal gradient step~\cite{Kim2015}.
Although numerically the algorithm so obtained behaved well, no theoretical
convergence proof was provided.

In the present paper we prove convergence of algorithms of the form
\begin{equation}
\mathbf{x}_{k}=P_{L_{k}}\left(\mathbf{y}_{k}\right),\label{eq:PL(y)}
\end{equation}
\begin{equation}
\mathbf{y}_{k+1}=\mathbf{x}_{k}+\frac{t_{k}-1}{t_{k+1}}\left(\mathbf{x}_{k}-\mathbf{x}_{k-1}\right)+\frac{t_{k}}{t_{k+1}}\left(\eta_{k}-1\right)\left(\mathbf{x}_{k}-\mathbf{y}_{k}\right).\label{eq:y-update}
\end{equation}
Under suitable conditions on the sequence $\{\eta_{k}\}\subset\mathbb{R}$,
convergence can be proven to follow
\begin{equation}
\Psi(\mathbf{x}_{k})-\Psi(\mathbf{x}^{*})\leq\frac{2L_{k}\left\Vert \mathbf{x}_{0}-\mathbf{x}^{*}\right\Vert ^{2}}{\eta_{k}(k+1)^{2}}.
\end{equation}
Notice that for $\eta_{k}=2$ the $\mathbf{y}_{k}$ update in~\eqref{eq:y-update}
is the same than the one of OGM and the above convergence bound also
equals the one for OGM. It appears not to be possible, however, to
prove convergence of algorithm~\eqref{eq:PL(y)}-\eqref{eq:y-update}
with $\eta_{k}\equiv2$ for every problem~\cite{thg17}. Instead,
in our approach a valid range of values for the coefficient $\eta_{k}$
is computed during the execution of iteration $k$ of the algorithm,
but we noticed that in practice the upper bound of this range is often
larger than $2$, allowing, for example, us to often use $\eta_{k}\approx2$
and thus to make \eqref{eq:y-update} approximately equal to~\eqref{eq:y-update-OGM}.

A similar method is the generalization of FISTA called OMFISTA~\cite{Yamagishi2011}.
The convergence theory provided for OMFISTA was valid only in the
range $\eta_{k}\in(0,1]$, and so did not include the case $\eta_{k}=2$
that characterizes OISTA. In \cite{Zibetti2019a}, monotone FISTA
with variable acceleration (MFISTA-VA) is proposed, and it is shown
that an $\eta_{k}>2$ may be chosen when the new iterate $\mathbf{x}_{k}$
satisfies $\Psi\left(\mathbf{x}_{k}\right)<\Psi\left(P_{L_{k}}\left(\mathbf{y}_{k}\right)\right)$,
which may lead to an extra algorithmic step and variable acceleration.
Here we generalize what is proposed in \cite{Zibetti2019a} and prove
that similar bounds can be achieved when \eqref{eq:PL(y)} is used,
thereby making our results relevant to the convergence of OISTA.

All of the above approaches speed up the Iterative Soft Thresholding
Algorithm (ISTA), which has iterations of the form $\mathbf{x}_{k+1}=P_{L}\left(\mathbf{x}_{k}\right)$,
using a momentum term that considers previous iterates when updating
the current iterate. Another such approach, based on a different insight
than Nesterov's, is the Two-Step Iterative Soft Thresholding (TwIST)~\cite{bif07}
method. TwIST can be applied to problems of the form
\begin{equation}
\min\quad\frac{1}{2}\|K\mathbf{x}-\boldsymbol{b}\|^{2}+\phi(\mathbf{x})\label{eq: Twist}
\end{equation}
and is considerably faster than IST in the case of poorly conditioned
invertible operators $K$. TwIST is inspired by a two-step method
for the solution of the linear system of equations $K\boldsymbol{x}=\boldsymbol{b}$
and is obtained by replacing appearances of a scaled version of the
gradient $\nabla f$ of $f(\boldsymbol{x})=\frac{1}{2}\|K\mathbf{x}-\boldsymbol{b}\|^{2}$
in that algorithm by the proximal-gradient operator $P_{L}$.

SpaRSA (Sparse Reconstruction by Separable Approximation)~\cite{wnf09}
is another approach for improving over IST. In this case, the smooth
part $f$ of problem~\eqref{eq:problem} is not assumed to be convex
and the key ingredient for the algorithm's efficiency seems to be
a nonmonotone line-search procedure inspired by the seminal approach
by Grippo, Lampariello and Lucidi (GLL)~\cite{gll86}. One interesting
question regarding SpaRSA is whether the GLL line search could be
replaced in SpaRSA by the nonmonontone line search of Zhang and Hagher
(ZH)~\cite{zhh04}, which in many cases compares favorably to GLL
and whether these benefits of ZH would carry over to SpaRSA.

Yet another approach for generalizing and improving IST is the Generalized
Iterative Shrinkage and Thresholding (GIST) algorithm~\cite{gzl13}.
The theory behind GIST convergence results removes the convexity hypothesis
over $f$ and relax that over $\phi$ in~\eqref{eq:problem} by supposing
instead that $\phi$ can be written as a difference of two convex
functions. This latter assumption generalizes considerably over the
convexity hypothesis usually imposed to the non-smooth part in IST-inspired
methods while retaining important theoretical tools such as the existence
of subderivatives and, consequently, the existence of useful necessary
optimality conditions for this function.

\section{The New Methods \label{sec:The-New-Method}}

In the current section we present our proposed algorithms in detail.
We start by describing an abstract method, which we name Non-Deterministic
PGM (NDPGM).

\begin{algorithm}
\caption{NDPGM}
\label{PGM_abstract}

\begin{algorithmic}[1]

\STATE{\textbf{input }$t_{1}\geq1$, $\mathbf{y}_{1}\in\mathbb{R}^{n}$,
$N\in\mathbb{N}$}

\STATE{\textbf{for }$k=1,...,N$}

\STATE{~~~~\textbf{set }$L_{k}\in\mathbb{R}^{n}$}

\STATE{~~~~\textbf{set }$\mathbf{z}_{k}=P_{L_{k}}\left(\mathbf{y}_{k}\right)$}

\STATE{~~~~\textbf{set }$\mathbf{x}_{k}\in\mathbb{R}^{n}$} 

\STATE{~~~~\textbf{set }$\eta_{k}\in\mathbb{R}$}

\STATE{~~~~\textbf{set $t_{k+1}=\frac{1+\sqrt{1+4t_{k}^{2}}}{2}$}} 

\STATE{~~~~\textbf{set $\mathbf{y}_{k+1}=\mathbf{x}_{k}+\frac{t_{k}-1}{t_{k+1}}\left(\mathbf{x}_{k}-\mathbf{x}_{k-1}\right)+\frac{t_{k}}{t_{k+1}}\left(\mathbf{z}_{k}-\mathbf{x}_{k}\right)+\frac{t_{k}}{t_{k+1}}\left(\eta_{k}-1\right)\left(\mathbf{z}_{k}-\mathbf{y}_{k}\right)$}} 

\end{algorithmic}
\end{algorithm}

Note the non-deterministic nature of Steps 3, 5 and 6 of NDPGM. This
leaves some important sequences unspecified, including the sequence
of iterates $\left\{ \mathbf{x}_{k}\right\} \subset\mathbb{R}^{n}$
itself. This is done on purpose so that NDPGM is a model for several
methods. As two examples, notice how both FISTA~\cite{Beck2009}
and MFISTA~\cite{Beck2009a}, described respectively in Algorithm~\ref{algo:FISTA}
and Algorithm~\ref{algo:MFISTA}, fit this description. In both cases
$\eta_{k}\equiv1$ is used and a backtracking procedure is performed
at Steps~\ref{step:fista_backtracking_loop}~and~\ref{step:fista_backtracking_increase}
of both algorithms in order to determine the parameter $L_{k}$ at
each iteration, measuring the difference between the function value
$\Psi\left(P_{L_{k}}\left(\mathbf{y}_{k}\right)\right)$ and $Q_{L_{k}}\left(P_{L_{k}}\left(\mathbf{y}_{k}\right),\mathbf{y}_{k}\right)$
where, for $L\in\mathbb{R}$ and $\mathbf{x}$, $\mathbf{y}\in\mathbb{R}^{n}$,
$Q_{L}(\mathbf{x},\mathbf{y})$ is the surrogate function:
\begin{equation}
Q_{L}(\mathbf{x},\mathbf{y}):=f(\mathbf{y})+\left\langle \nabla f(\mathbf{y}),\mathbf{x}-\mathbf{y}\right\rangle +\frac{L}{2}\left\Vert \mathbf{x}-\mathbf{y}\right\Vert ^{2}+\phi(\mathbf{x}).\label{eq:surrogate}
\end{equation}
The difference between these methods is that in FISTA we have $\mathbf{x}_{k}=\mathbf{z}_{k}$,
while MFISTA uses $\mathbf{x}_{k}=\argmin_{\mathbf{x}\in\left\{ \mathbf{z}_{k},\mathbf{x}_{k-1}\right\} }\Psi(\mathbf{x})$.
The MFISTA-VA of \cite{Zibetti2019a} is also an instance of Algorithm~\ref{PGM_abstract},
in this case with $\eta_{k}\geq1$.

\begin{algorithm}
\caption{FISTA}
\label{algo:FISTA}

\begin{algorithmic}[1]

\STATE{\textbf{input }$L_{0}>0$, $\beta>1$, $\mathbf{x}_{0}\in\mathbb{R}^{n}$,
$N\in\mathbb{N}$}

\STATE{\textbf{set }$t_{1}=1$, $\mathbf{y}_{1}=\mathbf{x}_{0}$}

\STATE{\textbf{for }$k=1,...,N$}

\STATE{~~~~\textbf{set }$L_{k}=L_{k-1}$}

\STATE{~~~~\textbf{while }$\Psi\left(P_{L_{k}}\left(\mathbf{y}_{k}\right)\right)>Q_{L_{k}}\left(P_{L_{k}}\left(\mathbf{y}_{k}\right),\mathbf{y}_{k}\right)$}\label{step:fista_backtracking_loop}

\STATE{~~~~~~~~\textbf{set $L_{k}=\beta L_{k}$}}\label{step:fista_backtracking_increase}

\STATE{~~~~\textbf{set }$\mathbf{x}_{k}=P_{L_{k}}\left(\mathbf{y}_{k}\right)$}

\STATE{~~~~\textbf{set $t_{k+1}=\frac{1+\sqrt{1+4t_{k}^{2}}}{2}$}} 

\STATE{~~~~\textbf{set $\mathbf{y}_{k+1}=\mathbf{x}_{k}+\frac{t_{k}-1}{t_{k+1}}\left(\mathbf{x}_{k}-\mathbf{x}_{k-1}\right)$}} 

\end{algorithmic}
\end{algorithm}

\begin{algorithm}
\caption{MFISTA}
\label{algo:MFISTA}

\begin{algorithmic}[1]

\STATE{\textbf{input }$L_{0}>0$, $\beta>1$, $\mathbf{x}_{0}\in\mathbb{R}^{n}$,
$N\in\mathbb{N}$}

\STATE{\textbf{set }$t_{1}=1$, $\mathbf{y}_{1}=\mathbf{x}_{0}$}

\STATE{\textbf{for }$k=1,...,N$}

\STATE{~~~~\textbf{set }$L_{k}=L_{k-1}$}

\STATE{~~~~\textbf{while }$\Psi\left(P_{L_{k}}\left(\mathbf{y}_{k}\right)\right)>Q_{L_{k}}\left(P_{L_{k}}\left(\mathbf{y}_{k}\right),\mathbf{y}_{k}\right)$}

\STATE{~~~~~~~~\textbf{set $L_{k}=\beta L_{k}$}}

\STATE{~~~~\textbf{set }$\mathbf{z}_{k}=P_{L_{k}}\left(\mathbf{y}_{k}\right)$}

\STATE{~~~~\textbf{set }$\mathbf{x}_{k}=\argmin_{\mathbf{x}\in\left\{ \mathbf{z}_{k},\mathbf{x}_{k-1}\right\} }\Psi(\mathbf{x})$}

\STATE{~~~~\textbf{set $t_{k+1}=\frac{1+\sqrt{1+4t_{k}^{2}}}{2}$}} 

\STATE{~~~~\textbf{set $\mathbf{y}_{k+1}=\mathbf{x}_{k}+\frac{t_{k}-1}{t_{k+1}}\left(\mathbf{x}_{k}-\mathbf{x}_{k-1}\right)+\frac{t_{k}}{t_{k+1}}\left(\mathbf{z}_{k}-\mathbf{x}_{k}\right)$}} 

\end{algorithmic}
\end{algorithm}

Before presenting our methods, we define the following:
\begin{align}
\Delta_{a}(L,\mathbf{y}) & {}:=Q_{L}\left(P_{L}(\mathbf{y}),\mathbf{y}\right)-\Psi\left(P_{L}(\mathbf{y})\right)\nonumber \\
\Delta_{b}(\mathbf{x},\mathbf{y}) & {}:=f(\mathbf{x})-f(\mathbf{y})-\left\langle \nabla f(\mathbf{y}),\mathbf{x}-\mathbf{y}\right\rangle \label{eq:Deltas_definition}\\
\Delta_{c}(L,\mathbf{x},\mathbf{y}) & {}:=\phi(\mathbf{x})-\phi\left(P_{L}(\mathbf{y})\right)-\left\langle -\nabla f(\mathbf{y})-L(P_{L}(\mathbf{y})-\mathbf{y}),\mathbf{x}-P_{L}(\mathbf{y})\right\rangle .\nonumber 
\end{align}
Our proposed Fast Proximal Gradient Method (FPGM) is described in
Algorithm~\ref{algo:FPGM}.

\begin{algorithm}
\caption{FPGM}
\label{algo:FPGM}

\begin{algorithmic}[1]

\STATE{\textbf{input }$t_{1}\geq1$, $L_{0}>0$, $\beta>1$, $\mathbf{x}_{0}\in\mathbb{R}^{n}$,
$K\in\mathbb{N}$, $\overline{\eta}\in[1,\infty)\cap\infty$, $N\in\mathbb{N}$}

\STATE{\textbf{set }$\eta_{0}=\overline{\eta}$}

\STATE{\textbf{set }$\mathbf{y}_{1}=\mathbf{x}_{0}$}

\STATE{\textbf{for }$k=1,...,N$}

\STATE{~~~~\textbf{set }$L_{k}=L_{k-1}$}

\STATE{~~~~\textbf{while }$\Psi\left(P_{L_{k}}\left(\mathbf{y}_{k}\right)\right)>Q_{L_{k}}\left(P_{L_{k}}\left(\mathbf{y}_{k}\right),\mathbf{y}_{k}\right)$}

\STATE{~~~~~~~~\textbf{set $L_{k}=\beta L_{k}$}}

\STATE{~~~~\textbf{set }$\mathbf{x}_{k}=P_{L_{k}}\left(\mathbf{y}_{k}\right)$}

\STATE{~~~~\textbf{set }$\gamma_{k}=1+2\frac{\Delta_{a}\left(L_{k},\mathbf{y}_{k}\right)+\left(1-1/t_{k}\right)\bigl(\Delta_{b}\left(\mathbf{x}_{k-1},\mathbf{y}_{k}\right)+\Delta_{c}\left(L_{k},\mathbf{x}_{k-1},\mathbf{y}_{k}\right)\bigr)}{L_{k}\left\Vert \mathbf{z}_{k}-\mathbf{y}_{k}\right\Vert ^{2}}$}\label{line:gamma_computation}

\STATE{~~~~\textbf{if }$k\leq K$}

\STATE{~~~~~~~~\textbf{set }$\eta_{k}=\min\left\{ \gamma_{k},\overline{\eta}\right\} $}

\STATE{~~~~\textbf{else}}

\STATE{~~~~~~~~\textbf{set }$\eta_{k}=\min\left\{ \gamma_{k},\eta_{k-1}L_{k}/L_{k-1},\overline{\eta}\right\} $}\label{line:eta_FPGM}

\STATE{~~~~\textbf{set $t_{k+1}=\frac{1+\sqrt{1+4t_{k}^{2}}}{2}$}} 

\STATE{~~~~\textbf{set $\mathbf{y}_{k+1}=\mathbf{x}_{k}+\frac{t_{k}-1}{t_{k+1}}\left(\mathbf{x}_{k}-\mathbf{x}_{k-1}\right)+\frac{t_{k}}{t_{k+1}}\left(\eta_{k}-1\right)\left(\mathbf{x}_{k}-\mathbf{y}_{k}\right)$}} 

\end{algorithmic}
\end{algorithm}

Notice how the computation of $\Delta_{a}\left(L_{k,}\mathbf{y}_{k}\right)$,
$\Delta_{b}\left(\mathbf{x}_{k-1},\mathbf{y}_{k}\right)$, and $\Delta_{c}\left(L_{k,}\mathbf{x}_{k-1,}\mathbf{y}_{k}\right)$
uses only gradients and function evaluations that would have already
been computed during the execution of the algorithm, except for the
$\phi$-values in $\Delta_{c}\left(L_{k},\mathbf{x}_{k-1},\mathbf{y}_{k}\right))$.
If this is costly, the user may simply replace $\Delta_{c}\left(L_{k},\mathbf{x}_{k-1},\mathbf{y}_{k}\right)$
computation by zero because $\Delta_{c}\left(L,\mathbf{x},\mathbf{y}\right)\geq0$,
according to~\eqref{eq:Delta_c_nonneg} in the next section, and
the value used for $\eta_{k}$ in Step~\ref{line:eta_FPGM} of Algorithm~\ref{algo:FPGM}
is an upper bound for the valid range of values for $\eta_{k}$ that
ensure convergence, but smaller values can also be used. The monotone
version of the FPGM, called MFPGM is given in Algorithm~\ref{algo:MFPGM}.
The monotone versions of the FOMs have not to date been analyzed through
the algorithm optimization approach of~\cite{drt14,Kim2016c} and
the theory for these methods is known from the analysis of the more
general monotone proximal gradient algorithms.

\begin{algorithm}
\caption{MFPGM}
\label{algo:MFPGM}

\begin{algorithmic}[1]

\STATE{\textbf{input }$t_{1}\geq1$, $L_{0}>0$, $\beta>1$, $\mathbf{x}_{0}\in\mathbb{R}^{n}$,
$K\in\mathbb{N}$, $\overline{\eta}\in[1,\infty)\cap\infty$, $N\in\mathbb{N}$}

\STATE{\textbf{set }$\eta_{0}=\overline{\eta}$}

\STATE{\textbf{set }$\mathbf{y}_{1}=\mathbf{x}_{0}$}

\STATE{\textbf{for }$k=1,...,N$}

\STATE{~~~~\textbf{set }$L_{k}=L_{k-1}$}

\STATE{~~~~\textbf{while }$\Psi\left(P_{L_{k}}(\mathbf{y}_{k})\right)>Q_{L_{k}}\left(P_{L_{k}}(\mathbf{y}_{k}),\mathbf{y}_{k}\right)$}

\STATE{~~~~~~~~\textbf{set $L_{k}=\beta L_{k}$}}

\STATE{~~~~\textbf{set }$\mathbf{z}_{k}=P_{L_{k}}(\mathbf{y}_{k})$}

\STATE{~~~~\textbf{set }$\mathbf{x}_{k}=\argmin_{\mathbf{x}\in\{\mathbf{z}_{k},\mathbf{x}_{k-1}\}}\Psi(\mathbf{x})$}

\STATE{~~~~\textbf{set }$\gamma_{k}=1+2\frac{\Delta_{a}(L_{k},\mathbf{y}_{k})+(1-1/t_{k})\bigl(\Delta_{b}(\mathbf{x}_{k-1},\mathbf{y}_{k})+\Delta_{c}(L_{k},\mathbf{x}_{k-1},\mathbf{y}_{k})\bigr)+\bigl(\Psi(\mathbf{z}_{k})-\Psi(\mathbf{x}_{k})\bigr)}{L_{k}\left\Vert \mathbf{z}_{k}-\mathbf{y}_{k}\right\Vert ^{2}}$}

\STATE{~~~~\textbf{if }$k\leq K$}

\STATE{~~~~~~~~\textbf{set }$\eta_{k}=\min\{\gamma_{k},\overline{\eta}\}$}

\STATE{~~~~\textbf{else}}

\STATE{~~~~~~~~\textbf{set }$\eta_{k}=\min\{\gamma_{k},\eta_{k-1}L_{k}/L_{k-1},\overline{\eta}\}$}

\STATE{~~~~\textbf{set $t_{k+1}=\frac{1+\sqrt{1+4t_{k}^{2}}}{2}$}} 

\STATE{~~~~\textbf{set $\mathbf{y}_{k+1}=\mathbf{x}_{k}+\frac{t_{k}-1}{t_{k+1}}\left(\mathbf{x}_{k}-\mathbf{x}_{k-1}\right)+\frac{t_{k}}{t_{k+1}}\left(\mathbf{z}_{k}-\mathbf{x}_{k}\right)+\frac{t_{k}}{t_{k+1}}\left(\eta_{k}-1\right)\left(\mathbf{z}_{k}-\mathbf{y}_{k}\right)$}} 

\end{algorithmic}
\end{algorithm}

\section{Convergence Analysis\label{sec:Convergence-Analysis}}

We start with an introductory lemma, which will be repeatedly used
in the remaining of the convergence analysis.
\begin{lemma}
\label{lem:key_lemma}Let $\mathbf{x}$, $\mathbf{y}\in\mathbb{R}^{n},$
$\Psi:\mathbb{R}^{n}\to\mathbb{R}$ as in~\eqref{eq:problem}, $P_{L}:\mathbb{R}^{n}\to\mathbb{R}^{n}$
as in~\eqref{eq:prox_grad}, and $0<L\in\mathbb{R}$. Then
\begin{equation}
\Psi(\mathbf{x})-\Psi(P_{L}(\mathbf{y}))=\frac{L}{2}\left\Vert P_{L}(\mathbf{y})-\mathbf{y}\right\Vert ^{2}+L\left\langle P_{L}(\mathbf{y})-\mathbf{y},\mathbf{y}-\mathbf{x}\right\rangle +\Delta(L,\mathbf{x},\mathbf{y})
\end{equation}
with
\begin{equation}
\Delta(L,\mathbf{x},\mathbf{y}):=\Delta_{a}(L,\mathbf{y})+\Delta_{b}(\mathbf{x},\mathbf{y})+\Delta_{c}(L,\mathbf{x},\mathbf{y}),\label{eq:Delta_definition}
\end{equation}
where $\Delta_{a}(L,\mathbf{y})$, $\Delta_{b}(\mathbf{x},\mathbf{y})$,
and $\Delta_{c}(L,\mathbf{x},\mathbf{y})$ are defined in~\eqref{eq:Deltas_definition}.\end{lemma}
\begin{proof}
Notice that, because of~\eqref{eq:surrogate} and the first definition
in~\eqref{eq:Deltas_definition}, we have 
\begin{equation}
\Psi\left(P_{L}(\mathbf{y})\right)=f(\mathbf{y})+\left\langle \nabla f(\mathbf{y}),P_{L}(\mathbf{y})-\mathbf{y}\right\rangle +\frac{L}{2}\left\Vert P_{L}(\mathbf{y})-\mathbf{y}\right\Vert ^{2}+\phi\left(P_{L}(\mathbf{y})\right)-\Delta_{a}(L,\mathbf{y}),
\end{equation}
Therefore,
\begin{multline}
\Psi\left(P_{L}(\mathbf{y})\right)=f(\mathbf{x})+\left\langle \nabla f(\mathbf{y})+\phi'\left(P_{L}(\mathbf{y})\right),P_{L}(\mathbf{y})-\mathbf{x}\right\rangle +\frac{L}{2}\left\Vert P_{L}(\mathbf{y})-\mathbf{y}\right\Vert ^{2}+{}\\
\phi\left(\mathbf{x}\right)-\Delta_{a}(L,\mathbf{y})-\Delta_{b}(\mathbf{x},\mathbf{y})-{}\\
\left(\phi(\mathbf{x})-\phi\left(P_{L}(\mathbf{y})\right)-\left\langle \phi'\left(P_{L}(\mathbf{y})\right),\mathbf{x}-P_{L}(\mathbf{y})\right\rangle \right),
\end{multline}
where $\phi'\left(P_{L}(\mathbf{y})\right)=-L(P_{L}(\mathbf{y})-\mathbf{y})-\nabla f(\mathbf{y})$
is a subgradient of $\phi$ at $P_{L}(\mathbf{y})$ according to the
necessary and sufficient optimality conditions for~\eqref{eq:prox_grad}
(we refer the reader to~\cite{hil93} for the definition and properties
of a subgradient of a convex function, and also for optimality conditions
for unconstrained convex minimization problems). Then
\begin{equation}
0\leq\phi(\mathbf{x})-\phi\left(P_{L}(\mathbf{y})\right)-\left\langle \phi'\left(P_{L}(\mathbf{y})\right),\mathbf{x}-P_{L}(\mathbf{y})\right\rangle =\Delta_{c}(L,\mathbf{x},\mathbf{y})\label{eq:Delta_c_nonneg}
\end{equation}
and we have
\begin{multline}
\Psi\left(P_{L}(\mathbf{y})\right)=f(\mathbf{x})+\left\langle \nabla f(\mathbf{y})+\phi'\left(P_{L}(\mathbf{y})\right),P_{L}(\mathbf{y})-\mathbf{x}\right\rangle +\frac{L}{2}\left\Vert P_{L}(\mathbf{y})-\mathbf{y}\right\Vert ^{2}+{}\\
\phi\left(\mathbf{x}\right)-\Delta_{a}(L,\mathbf{y})-\Delta_{b}(\mathbf{x},\mathbf{y})-\Delta_{c}(L,\mathbf{x},\mathbf{y}).
\end{multline}
Rearranging we get
\begin{equation}
\Psi(\mathbf{x})-\Psi\left(P_{L}(\mathbf{y})\right)=\left\langle L(P_{L}(\mathbf{y})-\mathbf{y}),P_{L}(\mathbf{y})-\mathbf{x}\right\rangle -\frac{L}{2}\left\Vert P_{L}(\mathbf{y})-\mathbf{y}\right\Vert ^{2}+\Delta(L,\mathbf{x},\mathbf{y}),
\end{equation}
and then 
\begin{equation}
\Psi(\mathbf{x})-\Psi\left(P_{L}(\mathbf{y})\right)=L\left\langle P_{L}(\mathbf{y})-\mathbf{y},\mathbf{y}-\mathbf{x}\right\rangle +\frac{L}{2}\left\Vert P_{L}(\mathbf{y})-\mathbf{y}\right\Vert ^{2}+\Delta(L,\mathbf{x},\mathbf{y}).
\end{equation}

\end{proof}

The next result is our main convergence theorem. It will be proven
under the general condition~\eqref{eq:convergence_condition} on
the parameters $L_{k}$ and $\eta_{k}$. This condition may involve
the optimal point and is likely not verifiable in practice. Later
we will show how to guarantee validity of~\eqref{eq:convergence_condition}
through more concrete conditions that can be verified in a practical
setting.
\begin{theorem}
Fix $\mathbf{x}^{*}\in\mathbb{R}^{n}$ and let $\Delta(L,\mathbf{x},\mathbf{y})$
be as in~\eqref{eq:Delta_definition}, then Algorithm \ref{PGM_abstract}
under conditions\label{lem:convergence}
\begin{multline}
(\eta_{k}-1)t_{k}^{2}\left\Vert \mathbf{z}_{k}-\mathbf{y}_{k}\right\Vert ^{2}\leq\frac{2}{L_{k}}\Bigl(t_{k}(t_{k}-1)\Delta(L_{k},\mathbf{x}_{k-1},\mathbf{y}_{k})+{}\\
t_{k}\Delta(L_{k},\mathbf{x}^{*},\mathbf{y}_{k})+t_{k}^{2}\bigl(\Psi(\mathbf{z}_{k})-\Psi(\mathbf{x}_{k})\bigr)\Bigr)\label{eq:convergence_condition}
\end{multline}
or, equivalently,
\begin{multline}
\eta_{k}\leq1+\frac{2}{L_{k}\left\Vert \mathbf{z}_{k}-\mathbf{y}_{k}\right\Vert ^{2}}\Bigl((1-1/t_{k})\Delta(L_{k},\mathbf{x}_{k-1},\mathbf{y}_{k})+{}\\
(1/t_{k})\Delta(L_{k},\mathbf{x}^{*},\mathbf{y}_{k})+\bigl(\Psi(\mathbf{z}_{k})-\Psi(\mathbf{x}_{k})\bigr)\Bigr)\label{eq:convergence_condition_eta}
\end{multline}
and
\begin{equation}
\frac{\eta_{k+1}}{L_{k+1}}\leq\frac{\eta_{k}}{L_{k}}\label{eq:decreasing_steps}
\end{equation}
satisfies
\begin{multline}
\Psi(\mathbf{x}_{k})-\Psi(\mathbf{x}^{*})\leq{}\\
\dfrac{2L_{k}}{\eta_{k}\left(k+2t_{1}-1\right)^{2}}\biggl(\frac{2\eta_{1}t_{1}^{2}\left(\Psi(\mathbf{x}_{1})-\Psi(\mathbf{x}^{*})\right)}{L_{1}}+{}\\
\|\eta_{1}t_{1}\left(\mathbf{z}_{1}-\mathbf{y}_{1}\right)+t_{1}\mathbf{y}_{1}-(t_{1}-1)\mathbf{x}_{0}-\mathbf{x}^{*}\|^{2}\biggr).\label{eq:bound_ugly}
\end{multline}
\end{theorem}
\begin{proof}
Using Lemma~\ref{lem:key_lemma} with $L=L_{k+1}$, $\mathbf{x}=\mathbf{x}_{k}$,
and $\mathbf{y}=\mathbf{y}_{k+1}$, we obtain, because $\mathbf{z}_{k+1}=P_{L_{k+1}}(\mathbf{y}_{k+1})$,
\begin{align}
\Psi(\mathbf{x}_{k})-\Psi(\mathbf{x}_{k+1}) & {}=\Psi(\mathbf{x}_{k})-\Psi(\mathbf{z}_{k+1})+\bigl(\Psi(\mathbf{z}_{k+1})-\Psi(\mathbf{x}_{k+1})\bigr)\nonumber \\
{} & =\frac{L_{k+1}}{2}\left\Vert \mathbf{z}_{k+1}-\mathbf{y}_{k+1}\right\Vert ^{2}+{}\nonumber \\
 & \qquad\qquad L_{k+1}\left\langle \mathbf{z}_{k+1}-\mathbf{y}_{k+1},\mathbf{y}_{k+1}-\mathbf{x}_{k}\right\rangle +{}\nonumber \\
 & \qquad\qquad\qquad\quad\Delta(L_{k+1},\mathbf{x}_{k},\mathbf{y}_{k+1})+\bigl(\Psi(\mathbf{z}_{k+1})-\Psi(\mathbf{x}_{k+1})\bigr).
\end{align}
Therefore,
\begin{multline}
\frac{2}{L_{k+1}}\left(d_{k}-d_{k+1}\right)=\left\Vert \mathbf{z}_{k+1}-\mathbf{y}_{k+1}\right\Vert ^{2}+2\left\langle \mathbf{z}_{k+1}-\mathbf{y}_{k+1},\mathbf{y}_{k+1}-\mathbf{x}_{k}\right\rangle +{}\\
\frac{2}{L_{k+1}}\Bigl(\Delta(L_{k+1},\mathbf{x}_{k},\mathbf{y}_{k+1})+\bigl(\Psi(\mathbf{z}_{k+1})-\Psi(\mathbf{x}_{k+1})\bigr)\Bigr),
\end{multline}
where $d_{k}=\Psi(\mathbf{x}_{k})-\Psi(\mathbf{x}^{*})$. Using Lemma~\ref{lem:key_lemma}
once more in the same manner, but now with $\mathbf{x}=\mathbf{x}^{*}$,
we have
\begin{multline}
\frac{2}{L_{k+1}}\left(-d_{k+1}\right)=\left\Vert \mathbf{z}_{k+1}-\mathbf{y}_{k+1}\right\Vert ^{2}+2\left\langle \mathbf{z}_{k+1}-\mathbf{y}_{k+1},\mathbf{y}_{k+1}-\mathbf{x}^{*}\right\rangle +{}\\
\frac{2}{L_{k+1}}\Bigl(\Delta(L_{k+1},\mathbf{x}^{*},\mathbf{y}_{k+1})+\bigl(\Psi(\mathbf{z}_{k+1})-\Psi(\mathbf{x}_{k+1})\bigr)\Bigr)
\end{multline}
Multiplying the above equations by $t_{k+1}(t_{k+1}-1)$ and $t_{k+1}$,
respectively, and adding the results we get

\begin{multline}
\frac{2}{L_{k+1}}\left(t_{k+1}(t_{k+1}-1)d_{k}-t_{k+1}^{2}d_{k+1}\right)=t_{k+1}^{2}\left\Vert \mathbf{z}_{k+1}-\mathbf{y}_{k+1}\right\Vert ^{2}+{}\\
2\left\langle t_{k+1}\left(\mathbf{z}_{k+1}-\mathbf{y}_{k+1}\right),t_{k+1}\mathbf{y}_{k+1}-(t_{k+1}-1)\mathbf{x}_{k}-\mathbf{x}^{*}\right\rangle +{}\\
\frac{2}{L_{k+1}}\Bigl(t_{k+1}(t_{k+1}-1)\Delta(L_{k+1},\mathbf{x}_{k},\mathbf{y}_{k+1})+{}\\
t_{k+1}\Delta(L_{k+1},\mathbf{x}^{*},\mathbf{y}_{k+1})+t_{k+1}^{2}\bigl(\Psi(\mathbf{z}_{k+1})-\Psi(\mathbf{x}_{k+1})\bigr)\Bigr).
\end{multline}
Now we use the fact that the algorithm satisfies~$t_{k+1}(t_{k+1}-1)=t_{k}^{2}$,
resulting in

\begin{multline}
\frac{2}{L_{k+1}}\left(t_{k}^{2}d_{k}-t_{k+1}^{2}d_{k+1}\right)=t_{k+1}^{2}\left\Vert \mathbf{z}_{k+1}-\mathbf{y}_{k+1}\right\Vert ^{2}+{}\\
2\left\langle t_{k+1}\left(\mathbf{z}_{k+1}-\mathbf{y}_{k+1}\right),t_{k+1}\mathbf{y}_{k+1}-(t_{k+1}-1)\mathbf{x}_{k}-\mathbf{x}^{*}\right\rangle \\
{}+\frac{2}{L_{k+1}}\Bigl(t_{k+1}(t_{k+1}-1)\Delta(L_{k+1},\mathbf{x}_{k},\mathbf{y}_{k+1})+{}\\
t_{k+1}\Delta(L_{k+1},\mathbf{x}^{*},\mathbf{y}_{k+1})+t_{k+1}^{2}\bigl(\Psi(\mathbf{z}_{k+1})-\Psi(\mathbf{x}_{k+1})\bigr)\Bigr).
\end{multline}
Then we apply the relation
\begin{equation}
\langle\mathbf{x},\mathbf{y}\rangle=\frac{1}{2\eta}(\|\eta\mathbf{x}+\mathbf{y}\|^{2}-\|\eta\mathbf{x}\|^{2}-\|\mathbf{y}\|^{2})\label{eq:inner_2_norms}
\end{equation}
to obtain
\begin{multline}
\frac{2}{L_{k+1}}\left(t_{k}^{2}d_{k}-t_{k+1}^{2}d_{k+1}\right)=t_{k+1}^{2}\left\Vert \mathbf{z}_{k+1}-\mathbf{y}_{k+1}\right\Vert ^{2}\\
{}+\frac{1}{\eta_{k+1}}\biggl\|\eta_{k+1}t_{k+1}\left(\mathbf{z}_{k+1}-\mathbf{y}_{k+1}\right)+t_{k+1}\mathbf{y}_{k+1}-(t_{k+1}-1)\mathbf{x}_{k}-\mathbf{x}^{*}\biggr\|^{2}\\
{}-\frac{1}{\eta_{k+1}}\left\Vert \eta_{k+1}t_{k+1}\left(\mathbf{z}_{k+1}-\mathbf{y}_{k+1}\right)\right\Vert ^{2}-\frac{1}{\eta_{k+1}}\left\Vert t_{k+1}\mathbf{y}_{k+1}-(t_{k+1}-1)\mathbf{x}_{k}-\mathbf{x}^{*}\right\Vert ^{2}\\
{}+\frac{2}{L_{k+1}}\Bigl(t_{k+1}(t_{k+1}-1)\Delta(L_{k+1},\mathbf{x}_{k},\mathbf{y}_{k+1})+{}\\
t_{k+1}\Delta(L_{k+1},\mathbf{x}^{*},\mathbf{y}_{k+1})+t_{k+1}^{2}\bigl(\Psi(\mathbf{z}_{k+1})-\Psi(\mathbf{x}_{k+1})\bigr)\Bigr).\label{eq:main_eq_before_u}
\end{multline}
Let us then denote
\begin{equation}
\mathbf{u}_{k+1}:=\eta_{k+1}t_{k+1}\left(\mathbf{z}_{k+1}-\mathbf{y}_{k+1}\right)+t_{k+1}\mathbf{y}_{k+1}-(t_{k+1}-1)\mathbf{x}_{k}-\mathbf{x}^{*}
\end{equation}
and notice that $\mathbf{y}_{k+1}$ was defined in such a way that
\begin{equation}
\mathbf{u}_{k}=t_{k+1}\mathbf{y}_{k+1}-(t_{k+1}-1)\mathbf{x}_{k}-\mathbf{x}^{*}
\end{equation}
so that we can then rewrite~\eqref{eq:main_eq_before_u} as
\begin{multline}
\frac{2}{L_{k+1}}\left(t_{k}^{2}d_{k}-t_{k+1}^{2}d_{k+1}\right)=(1-\eta_{k+1})t_{k+1}^{2}\left\Vert \mathbf{z}_{k+1}-\mathbf{y}_{k+1}\right\Vert ^{2}+{}\\
\frac{1}{\eta_{k+1}}\left\Vert \mathbf{u}_{k+1}\right\Vert ^{2}-\frac{1}{\eta_{k+1}}\left\Vert \mathbf{u}_{k}\right\Vert ^{2}\\
{}+\frac{2}{L_{k+1}}\Bigl(t_{k+1}(t_{k+1}-1)\Delta(L_{k+1},\mathbf{x}_{k},\mathbf{y}_{k+1})+{}\\
t_{k+1}\Delta(L_{k+1},\mathbf{x}^{*},\mathbf{y}_{k+1})+t_{k+1}^{2}\bigl(\Psi(\mathbf{z}_{k+1})-\Psi(\mathbf{x}_{k+1})\bigr)\Bigr).
\end{multline}
Considering hypothesis~\eqref{eq:convergence_condition} we have
\begin{equation}
\frac{2}{L_{k+1}}\left(t_{k}^{2}d_{k}-t_{k+1}^{2}d_{k+1}\right)\geq\frac{1}{\eta_{k+1}}\left\Vert \mathbf{u}_{k+1}\right\Vert ^{2}-\frac{1}{\eta_{k+1}}\left\Vert \mathbf{u}_{k}\right\Vert ^{2},
\end{equation}
and then
\begin{equation}
\frac{2\eta_{k+1}}{L_{k+1}}\left(t_{k}^{2}d_{k}-t_{k+1}^{2}d_{k+1}\right)\geq\left\Vert \mathbf{u}_{k+1}\right\Vert ^{2}-\left\Vert \mathbf{u}_{k}\right\Vert ^{2}.
\end{equation}
Thus using the hypothesis $\eta_{k}/L_{k}\geq\eta_{k+1}/L_{k+1}$
:

\begin{equation}
\frac{2\eta_{k}t_{k}^{2}d_{k}}{L_{k}}-\frac{2\eta_{k+1}t_{k+1}^{2}d_{k+1}}{L_{k+1}}\geq\left\Vert \mathbf{u}_{k+1}\right\Vert ^{2}-\left\Vert \mathbf{u}_{k}\right\Vert ^{2}.\label{eq:decrease_k}
\end{equation}
Notice that denoting the above inequality as $a_{k}-a_{k+1}\ge b_{k+1}-b_{k}$
we can rearrange to \textbf{$a_{k+1}+b_{k+1}\leq a_{k}+b_{k}$} which,
because the sequences $\{a_{k}\}$ and $\{b_{k}\}$ are non-negative,
leads us to $a_{k}\leq a_{1}+b_{1}$, that is,
\begin{equation}
\frac{2\eta_{k}t_{k}^{2}d_{k}}{L_{k}}\leq\frac{2\eta_{1}t_{1}^{2}d_{1}}{L_{1}}+\|\mathbf{u}_{1}\|^{2}.
\end{equation}
That is
\begin{equation}
\Psi(\mathbf{x}_{k})-\Psi(\mathbf{x}^{*})\leq\frac{L_{k}}{2\eta_{k}t_{k}^{2}}\left(\frac{2\eta_{1}t_{1}^{2}d_{1}}{L_{1}}+\|\mathbf{u}_{1}\|^{2}\right).
\end{equation}
By noticing that $t_{k}\geq(k+2t_{1}-1)/2$ we get the desired result.
\end{proof}

Now we discuss practical ways of ensuring that~\eqref{eq:convergence_condition}
holds. A first step towards doing so is to consider Beck and Teboulle's
sufficient decrease criterion, which entails choosing at each iteration
$k$ of the algorithm a parameter $L_{k}$ such that
\begin{equation}
\Psi\left(P_{L_{k}}(\mathbf{y}_{k})\right)\leq Q_{L_{k}}\left(P_{L_{k}}(\mathbf{y}_{k}),\mathbf{y}_{k}\right).\label{eq:decrease_criterion}
\end{equation}
If such a criterion is used, it is straightforward to notice that
$\Delta(L_{k},\mathbf{x},\mathbf{y}_{k})$ is larger than or equal
$0$ for every $\mathbf{x}\in\mathbb{R}^{n}$ and, therefore, for
every $t_{1}\geq1$, condition~\eqref{eq:convergence_condition_eta}
is guaranteed to hold with $\eta_{k}\equiv1$, in which case we recover
the original FISTA when using $\mathbf{x}_{k}=\mathbf{z}_{k}$, as
in Algorithm~\ref{algo:FISTA}, and MFISTA when using $\mathbf{x}_{k}=\argmin_{\mathbf{x}\in\{\mathbf{z}_{k},\mathbf{x}_{k-1}\}}\Psi(\mathbf{x})$,
as in Algorithm~\ref{algo:MFISTA}. We are, however, interested in
the case where $\eta_{k}$ can be made considerably larger than $1$,
because it can yield faster convergence in practice.

The key difficulty in using condition~\eqref{eq:convergence_condition_eta}
is that we would have to be able to compute $\Delta(L_{k},\mathbf{x}^{*},\mathbf{y}_{k})=\Delta_{a}(L_{k},\mathbf{y}_{k})+\Delta_{b}(\mathbf{x}^{*},\mathbf{y}_{k})+\Delta_{c}(L_{k},\mathbf{x}^{*},\mathbf{y}_{k})$
which would involve an unknown optimizer $\mathbf{x}^{*}$. However,
from convexity of $f$ and $\phi,$ we know that $\Delta_{b}(\mathbf{x}^{*},\mathbf{y}_{k})\geq0$
and that $\Delta_{c}(L_{k},\mathbf{x}^{*},\mathbf{y}_{k})\geq0$.
Therefore, the following verifiable condition is sufficient 
\begin{multline}
\eta_{k}\leq1+\frac{2}{L_{k}\left\Vert \mathbf{z}_{k}-\mathbf{y}_{k}\right\Vert ^{2}}\bigl(\Delta_{a}(L_{k},\mathbf{y}_{k})+(1-1/t_{k})\Bigr(\Delta_{b}(\mathbf{x}_{k-1},\mathbf{y}_{k})+{}\\
\Delta_{c}(L_{k},\mathbf{x}_{k-1},\mathbf{y}_{k})\bigr)+\bigl(\Psi(\mathbf{z}_{k})-\Psi(\mathbf{x}_{k})\bigr)\Bigr).\label{eq:eta_practical_bound}
\end{multline}

We next prove the convergence of Algorithm~\ref{algo:GFPGM} below,
which is called General FPGM (GFPGM). The theoretical analysis uses
hypotheses that allow the inclusion as special cases, for example,
FPGM and its monotone version MFPGM, which are respectively Algorithms~\ref{algo:FPGM}
and~\ref{algo:MFPGM} above.

\begin{algorithm}
\caption{GFPGM}
\label{algo:GFPGM}

\begin{algorithmic}[1]

\STATE{\textbf{input }$t_{1}\geq1$, $L_{0}>0$, $\beta>1$, $\mathbf{y}_{1}$,
$\mathbf{x}_{0}\in\mathbb{R}^{n}$, $K\in\mathbb{N}$, $\overline{\eta}\in[1,\infty)\cap\infty$,
$N\in\mathbb{N}$}

\STATE{\textbf{set }$\eta_{0}=\overline{\eta}$}

\STATE{\textbf{for }$k=1,...,N$}

\STATE{~~~~\textbf{set }$L_{k}=L_{k-1}$}

\STATE{~~~~\textbf{while }$\Psi\left(P_{L_{k}}(\mathbf{y}_{k})\right)>Q_{L_{k}}\left(P_{L_{k}}(\mathbf{y}_{k}),\mathbf{y}_{k}\right)$}\label{line:decrease}

\STATE{~~~~~~~~\textbf{set $L_{k}=\beta L_{k}$}}\label{line:proximal_gradient}

\STATE{~~~~\textbf{set }$\mathbf{z}_{k}=P_{L_{k}}(\mathbf{y}_{k})$}

\STATE{~~~~\textbf{set }$\mathbf{x}_{k}\in\{\mathbf{x}:\Psi(\mathbf{x})\leq\Psi(\mathbf{z}_{k})\}$}

\STATE{~~~~\textbf{set }$\gamma_{k}=1+2\frac{\Delta_{a}(L_{k},\mathbf{y}_{k})+(1-1/t_{k})\bigl(\Delta_{b}(\mathbf{x}_{k-1},\mathbf{y}_{k})+\Delta_{c}(L_{k},\mathbf{x}_{k-1},\mathbf{y}_{k})\bigr)+\bigl(\Psi(\mathbf{z}_{k})-\Psi(\mathbf{x}_{k})\bigr)}{L_{k}\left\Vert \mathbf{z}_{k}-\mathbf{y}_{k}\right\Vert ^{2}}$}\label{line:gamma_k}

\STATE{~~~~\textbf{if }$k\leq K$ }

\STATE{~~~~~~~~\textbf{set }$\eta_{k}\in[1,\overline{\eta}]$}

\STATE{~~~~\textbf{else}}

\STATE{~~~~~~~~\textbf{set }$\eta_{k}\leq\min\{\gamma_{k},\eta_{k-1}L_{k}/L_{k-1},\overline{\eta}\}$}\label{line:eta_k}

\STATE{~~~~\textbf{set $t_{k+1}=\frac{1+\sqrt{1+4t_{k}^{2}}}{2}$}} 

\STATE{~~~~\textbf{set $\mathbf{y}_{k+1}=\mathbf{x}_{k}+\frac{t_{k}-1}{t_{k+1}}\left(\mathbf{x}_{k}-\mathbf{x}_{k-1}\right)+\frac{t_{k}}{t_{k+1}}\left(\mathbf{z}_{k}-\mathbf{x}_{k}\right)+\frac{t_{k}}{t_{k+1}}\left(\eta_{k}-1\right)\left(\mathbf{z}_{k}-\mathbf{y}_{k}\right)$}} 

\end{algorithmic}
\end{algorithm}

\begin{corollary}
Let $\{\mathbf{x}_{k}\}\subset\mathbb{R}^{n}$ be the sequence generated
by Algorithm~\ref{algo:GFPGM} with $K=0$, $t_{1}=1$ and $\mathbf{y}_{1}=\mathbf{x}_{0}$.
Fix some $\mathbf{x}^{*}\in\mathbb{R}^{n}$. Then we have
\begin{equation}
\Psi(\mathbf{x}_{k})-\Psi(\mathbf{x}^{*})\leq\frac{2L_{k}\|\mathbf{x}_{k}-\mathbf{x}^{*}\|^{2}}{\eta_{k}(k+1)^{2}}.\label{eq:bound_pretty}
\end{equation}
\label{thm:convergence_paricular}\end{corollary}
\begin{proof}
First we observe that the method indeed leaves Steps~\eqref{line:decrease}~and~\eqref{line:proximal_gradient}
of Algorithm~\ref{algo:GFPGM}. Let $\mathcal{L}$ be the Lipschitz
constant of $\nabla f$. Then, as it is well known, for $L\geq\mathcal{L}$,
the condition $Q_{L}(\mathbf{x},\mathbf{y})\geq\Psi(\mathbf{x})$
holds for every pair of vectors $\mathbf{x},$ $\mathbf{y}\in\mathbb{R}^{n}$.
Therefore, at iteration $k$, the loop of Steps~\eqref{line:decrease}~and~\eqref{line:proximal_gradient}
of Algorithm~\ref{algo:GFPGM} will be executed at most $K_{k}$
times where $K_{k}=\left\lceil \frac{\log\frac{\mathcal{L}}{L_{k-1}}}{\log\beta}\right\rceil $,
and where $\lceil x\rceil$ is the smaller integer larger than or
equal $x$. In fact, $\max\{L_{0},\beta\mathcal{L}\}$ is an upper
bound to $L_{k}$.

Notice that Steps~\ref{line:gamma_k}~and~\ref{line:eta_k} of
Algorithm~\ref{algo:GFPGM} ensure that conditions~\eqref{eq:convergence_condition_eta}~and~\eqref{eq:decreasing_steps}
are satisfied, so that Theorem~\ref{lem:convergence} can be applied
in this case leading to
\begin{multline}
\Psi(\mathbf{x}_{k})-\Psi(\mathbf{x}^{*})\leq\dfrac{2L_{k}}{\eta_{k}\left(k+2t_{1}-1\right)^{2}}\biggl(\frac{2\eta_{1}t_{1}^{2}\left(\Psi(\mathbf{x}_{1})-\Psi(\mathbf{x}^{*})\right)}{L_{1}}+{}\\
\|\eta_{1}t_{1}\left(\mathbf{z}_{1}-\mathbf{x}_{0}\right)+\mathbf{x}_{0}-\mathbf{x}^{*}\|^{2}\biggr).\label{eq:bound_aux}
\end{multline}
Now, we use Lemma~\ref{lem:key_lemma} with $L=L_{1}$, $\mathbf{x}=\mathbf{x}^{*}$,
$\mathbf{y}_{1}=\mathbf{x}_{0}$, and $\mathbf{y}=\mathbf{y}_{1}$
followed by~\eqref{eq:inner_2_norms} to obtain
\begin{align}
\Psi(\mathbf{x}_{1})-\Psi(\mathbf{x}^{*}) & {}=\Psi(\mathbf{z}_{1})-\Psi(\mathbf{x}^{*})+\left(\Psi(\mathbf{x}_{1})-\Psi(\mathbf{z}_{1})\right)\nonumber \\
 & {}=-\frac{L_{1}}{2}\|\mathbf{z}_{1}-\mathbf{y}_{1}\|^{2}-L_{1}\left<\mathbf{z}_{1}-\mathbf{y}_{1},\mathbf{y}_{1}-\mathbf{x}^{*}\right>-{}\nonumber \\
 & \qquad\qquad\qquad\qquad\qquad\qquad\quad\Delta(L_{1},\mathbf{x}^{*},\mathbf{y}_{1})+\left(\Psi(\mathbf{x}_{1})-\Psi(\mathbf{z}_{1})\right)\nonumber \\
 & {}=-\frac{L_{1}}{2}\|\mathbf{z}_{1}-\mathbf{x}_{0}\|^{2}-L_{1}\left<\mathbf{z}_{1}-\mathbf{x}_{0},\mathbf{x}_{0}-\mathbf{x}^{*}\right>-{}\nonumber \\
 & \qquad\qquad\qquad\qquad\qquad\qquad\quad\Delta(L_{1},\mathbf{x}^{*},\mathbf{y}_{1})+\left(\Psi(\mathbf{x}_{1})-\Psi(\mathbf{z}_{1})\right).
\end{align}
On the other hand,
\begin{multline}
\left<\mathbf{z}_{1}-\mathbf{x}_{0},\mathbf{x}_{0}-\mathbf{x}^{*}\right>=\frac{1}{2\eta_{1}t_{1}}\bigl(\|\eta_{1}t_{1}(\mathbf{z}_{1}-\mathbf{x}_{0})+\mathbf{x}_{0}-\mathbf{x}^{*}\|^{2}-{}\\
\|\eta_{1}t_{1}(\mathbf{z}_{1}-\mathbf{x}_{0})\|^{2}-\|\mathbf{x}_{0}-\mathbf{x}^{*}\|^{2}\bigr).
\end{multline}
Thus
\begin{align}
\frac{2\eta_{1}t_{1}^{2}\left(\Psi(\mathbf{x}_{1})-\Psi(\mathbf{x}^{*})\right)}{L_{1}} & {}=-\eta_{1}t_{1}^{2}\|\mathbf{z}_{1}-\mathbf{x}_{0}\|^{2}-t_{1}\|\eta_{1}t_{1}(\mathbf{z}_{1}-\mathbf{x}_{0})+\mathbf{x}_{0}-\mathbf{x}^{*}\|^{2}+{}\nonumber \\
 & \qquad\qquad\qquad t_{1}\|\eta_{1}t_{1}(\mathbf{z}_{1}-\mathbf{x}_{0})\|^{2}+t_{1}\|\mathbf{x}_{0}-\mathbf{x}^{*}\|^{2}+{}\nonumber \\
 & \qquad\qquad\qquad\qquad\frac{2\eta_{1}t_{1}^{2}}{L_{1}}\left(\Delta(L_{1},\mathbf{x}^{*},\mathbf{y}_{1})+\left(\Psi(\mathbf{x}_{1})-\Psi(\mathbf{z}_{1})\right)\right)\nonumber \\
 & {}=-\eta_{1}t_{1}^{2}\|\mathbf{z}_{1}-\mathbf{x}_{0}\|^{2}\biggl(1-\eta_{1}t_{1}-{}\nonumber \\
 & \qquad\qquad\qquad\frac{2\left(\Delta(L_{1},\mathbf{x}^{*},\mathbf{y}_{1})+\left(\Psi(\mathbf{x}_{1})-\Psi(\mathbf{z}_{1})\right)\right)}{L_{1}\|\mathbf{z}_{1}-\mathbf{x}_{0}\|^{2}}\biggr)-{}\\
 & \qquad\qquad\quad t_{1}\|\eta_{1}t_{1}(\mathbf{z}_{1}-\mathbf{x}_{0})+\mathbf{x}_{0}-\mathbf{x}^{*}\|^{2}+{}\nonumber \\
 & \qquad\qquad\qquad\qquad\qquad\qquad\qquad\qquad\qquad t_{1}\|\mathbf{x}_{0}-\mathbf{x}^{*}\|^{2}.
\end{align}
We simplify the result adopting the notation
\begin{equation}
\epsilon_{1}:=1-\eta_{1}t_{1}-\frac{2\left(\Delta(L_{1},\mathbf{x}^{*},\mathbf{y}_{1})+\left(\Psi(\mathbf{x}_{1})-\Psi(\mathbf{z}_{1})\right)\right)}{L_{1}\|\mathbf{z}_{1}-\mathbf{x}_{0}\|^{2}}
\end{equation}
finally reaching
\begin{multline}
\frac{2\eta_{1}t_{1}^{2}\left(\Psi(\mathbf{x}_{1})-\Psi(\mathbf{x}^{*})\right)}{L_{1}}+\|\eta_{1}t_{1}\left(\mathbf{z}_{1}-\mathbf{x}_{0}\right)+\mathbf{x}_{0}-\mathbf{x}^{*}\|^{2}={}\\
-\eta_{1}t_{1}^{2}\epsilon_{1}\|\mathbf{z}_{1}-\mathbf{x}_{0}\|^{2}+(1-t_{1})\|\eta_{1}t_{1}(\mathbf{z}_{1}-\mathbf{x}_{0})+\mathbf{x}_{0}-\mathbf{x}^{*}\|^{2}+t_{1}\|\mathbf{x}_{0}-\mathbf{x}^{*}\|^{2}.\label{eq:general_format}
\end{multline}
Now, notice that with $t_{1}=1$, Steps~\ref{line:gamma_k}~and~\ref{line:eta_k}
of Algorithm~\ref{algo:GFPGM} will ensure $\epsilon_{1}\geq0$ and,
therefore,
\begin{equation}
\frac{2\eta_{1}t_{1}^{2}\left(\Psi(\mathbf{x}_{1})-\Psi(\mathbf{x}^{*})\right)}{L_{1}}+\|\eta_{1}t_{1}\left(\mathbf{z}_{1}-\mathbf{x}_{0}\right)+\mathbf{x}_{0}-\mathbf{x}^{*}\|^{2}\leq\|\mathbf{x}_{0}-\mathbf{x}^{*}\|^{2}.
\end{equation}
This put into~\eqref{eq:bound_aux} gives the desired result.
\end{proof}

The main intent of the previous corollary is to ease the comparison
of bound~\eqref{eq:bound_ugly} of Theorem~\ref{lem:convergence}
with current knowledge about fast proximal gradient methods. It shows
that indeed, if it is possible to use $\eta_{k}\equiv2$ in OISTA,
this algorithm will share the same bound on the objective function
gap that has OGM and that $\eta_{k}>2$ may lead to an even better
constant in the convergence bound. Unfortunately, if $t_{1}=1$, it
is possible to show that $\eta_{1}\leq2$, and it can be seen that
$\eta_{1}=2$ is very unlikely to happen. In fact, $\eta_{1}$ can
be constrained this way to be no larger than $1$ if $\mathbf{x}_{1}=\mathbf{z}_{1}$
and $Q_{L_{1}}(\mathbf{z}_{1},\mathbf{y}_{1})=\Psi(\mathbf{z}_{1})$.
Then, if, as it is commonly the case, the step parameter $L_{k}$
never changes, we will permanently have to use $\eta_{k}<2$. On the
other hand, this is not an issue with $K>0$. Furthermore, $O(1/k^{2})$
convergence for Algorithm~\ref{algo:GFPGM} is still valid for $K>0$.
This is so because although Theorem~\ref{lem:convergence} cannot
be applied to Algorithm~\ref{algo:GFPGM} with $K>0$, one can still
apply Theorem~\ref{lem:convergence} to Algorithm~\ref{algo:GFPGM}
with $K=0$, but now restarted with $L_{0}=L_{K}$, $\mathbf{y}_{1}=\mathbf{y}_{K+1}$,
$t_{1}=t_{K+1}$, and $\mathbf{x}_{0}=\mathbf{\mathbf{x}}_{K}$.

\section{Experiments\label{sec:Experiments}}

We will divide the experimental section in two parts. The first of
these parts deals with high-resolution synchrotron-illuminated tomographic
image reconstruction of a biological sample. In this first part we
compare several fast proximal gradient algorithms from the viewpoint
of amount of objective function value reduction per computational
time in order to capture an idea of the relative algorithm performances
regarding numerical optimization speed.

Objective function value alone is not a measure of quality of image
reconstruction. This is why in the second set of experiments we use
simulated data from known mathematical images in order to compare
the reconstructions against a ground truth, which did not exist in
the previous case of the reconstruction of biological samples. The
comparisons use realistic task-oriented numerical figures of merit.
These experiments are repeated and statistical hypothesis testing
is used in order to assess the relevance of the outcome.

\begin{figure}
\includegraphics[width=0.51\columnwidth]{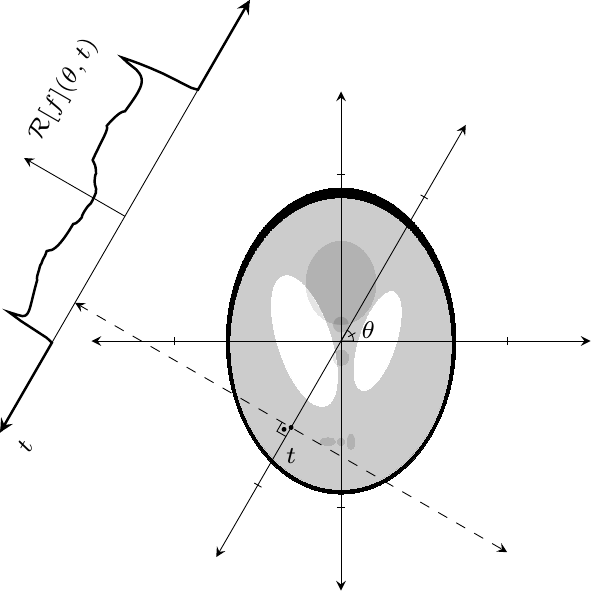}\includegraphics[width=0.49\columnwidth]{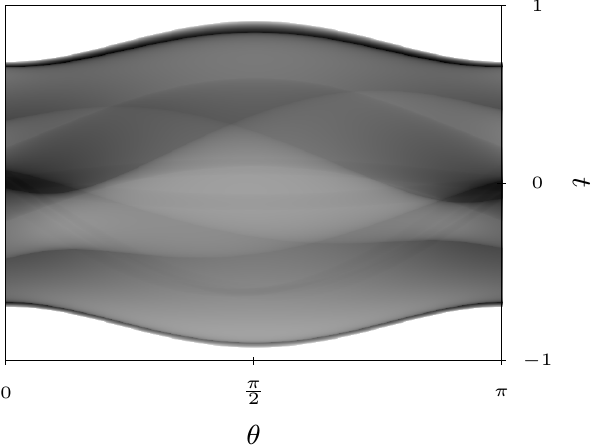}

\caption{Left: integration path for the Radon transform. There, $\theta$ is
the angle between the normal to the integration path and the horizontal
axis, and $t$ is the distance from the line of integration to the
origin. Right: gray-scale image in the $\theta\times t$ coordinate
system of the Radon transform of the image shown on the left.\label{fig:Radon}}
\end{figure}

Before heading to the experimental details, we explain generalities
of the tomographic image reconstruction problem. In tomography, the
goal is to reconstruct an image $\rho:\mathbb{R}^{2}\to\mathbb{R}$
from a finite number of approximate samples of its Radon transform
$\mathcal{R}[\rho]:[0,\pi)\times\mathbb{R}\to\mathbb{R}$, which is
defined as
\begin{equation}
\mathcal{R}\left[\rho\right](\theta,t):=\int_{-\infty}^{\infty}\rho\left(t\binom{\cos\theta}{\sin\theta}+s\binom{-\sin\theta}{\cos\theta}\right)\mathrm{d}s,
\end{equation}
see Figure~\ref{fig:Radon} for an illustration. Assuming that the
image belongs to a linear space spanned by a basis $\{\rho_{1},\rho_{2},\dots,\rho_{n}\}$
and because there is a finite number of samples $b_{i}\approx\mathcal{R}\left[\rho\right](\theta_{i},t_{i})$,
$i\in\{1,2,\dots,m\}$, the problem is then reduced to a linear system
of equations
\begin{equation}
R\mathbf{x}\approx\mathbf{b},
\end{equation}
where the unknowns $\mathbf{x}\in\mathbb{R}^{n}$ are the coefficients
of the linear combination that gives the reconstruction $\rho_{\text{recon}}:=\sum_{j=1}^{n}x_{j}\rho_{j}$.
In practice the meaning of the approximation above has to be made
clear, which can be done by, e.g., optimization. In this setting,
the constrained least squares
\begin{equation}
\mathbf{x}\in\argmin_{\mathbf{y}\in\mathbb{R}_{+}^{n}}\|R\mathbf{y}-\mathbf{b}\|^{2},
\end{equation}
is one among the many possible models. Because the tomographic reconstruction
problem is poorly conditioned, regularization is commonly applied
and the situation can in many cases be interpreted as the minimization
of the sum of a smooth convex function $f$ and a non-smooth convex
function $\phi$ as in~\eqref{eq:problem}, the problem which fast
proximal gradient methods are designed to solve. In theory, because
the Radon transform is a compact linear operator, its inverse is (moderately)
ill-posed~\cite{ehn00}. In practice, this ill-posedness can be problematic
in some circumstances. For example, in acquisitions where the signal
to noise ratio is below ideal such as when lower exposure to radiation
is desired or when the source is less predictable such as in emission
tomography. Beyond these cases, there might be reasons that make a
shorter acquisition time necessary (for example, for 4D tomography),
which might amplify the ill-posedness of the problem because higher
scanning speeds are usually obtained by sampling radially at a lower
rate.

\subsection{Constrained Maximum Likelihood Tomography with Synchrotron Data}

In synchrotron tomography, $\mathcal{R}\left[\rho\right](\theta_{i},t_{i})$
can be estimated by a computation involving $d_{i}$, $\omega_{i}$,
and $p_{i}$, where
\begin{itemize}
\item $d_{i}$ is a measurement of the dark field, which is the expected
number of detected events in the sensor during a scan without active
source and without the object between source and detector;
\item $\omega_{i}$ is a measurement of the flat field, which is the expected
number of detected events in the sensor during a scan with active
source and without the object between source and detector;
\item $p_{i}$ is the photon count, i.e., the number of events in the sensor
during a scan with active source and with the object between source
and detector.
\end{itemize}
Another idea is, instead of estimating $\mathcal{R}\left[\rho\right](\theta_{i},t_{i})$
from dark field, flat field, and photon count data, to use the probabilistically
inspired constrained model for transmission data from~\cite{erf99},
which we describe now. Let
\begin{equation}
f(\mathbf{x})=\sum_{i=1}^{m}h_{i}\left((R\mathbf{x})_{i}\right),
\end{equation}
where
\begin{equation}
h_{i}(b)=\omega_{i}e^{-b}+d_{i}-p_{i}\log(\omega_{i}e^{-b}+d_{i}).
\end{equation}
 Thus, the model consists of solving
\begin{equation}
\min_{\mathbf{x}\in\mathbb{R}_{+}^{n}}\quad f(\mathbf{x}),
\end{equation}
which is equivalent to
\begin{equation}
\min_{\mathbf{x}\in\mathbb{R}^{n}}\quad f(\mathbf{x})+\phi(\mathbf{x}),
\end{equation}
where $\phi(\mathbf{x})=\chi_{\mathbb{R}_{+}^{n}}(\mathbf{x})$, with
the indicator function $\chi_{C}$ of a convex closed set $C$ defined
as
\begin{equation}
\chi_{C}(\mathbf{x})=\begin{cases}
0 & \text{if}\quad\mathbf{x}\in C\\
\infty & \text{otherwise.}
\end{cases}
\end{equation}
The proximal operator of the indicator function is the Euclidean projection:
\begin{equation}
\prox_{\chi_{C}}(\mathbf{x})=\mathcal{P}_{C}(\mathbf{x})=\argmin_{\mathbf{y}\in C}\|\mathbf{x}-\mathbf{y}\|.
\end{equation}

\paragraph*{Data Geometry and Dimensions}

~Data was collected along $m_{r}=2048$ parallel lines on each one
of the $m_{v}=512$ views and the path between emitter and detector
for data point $i$ was parametrized by the pair $(\theta_{i},t_{i})$
for $i\in\left\{ 1,2,\dots,m_{r}m_{v}\right\} $ in the following
manner. First let us define $\kappa_{i}\in\left\{ 1,2,\dots,m_{v}\right\} $
as $\kappa_{i}:=1+\lfloor(i-1)/m_{r}\rfloor$ where $\lfloor x\rfloor$
is the largest integer smaller than or equal $x$ and $\ell_{i}\in\{1,2,\dots,m_{r}\}$
as $\ell_{i}:=1+(i-1)\%m_{r}$ where $a\%b$ is the remainder of the
integer division of $a\in\mathbb{N}$ by $b\in\mathbb{N}$. Then,
for $i\in\left\{ 1,2,\dots,m_{r}m_{v}\right\} $, we have

\begin{equation}
\theta_{i}=-\pi\frac{(\kappa_{i}-1)}{m_{v}-1}\quad\text{and}\quad t_{i}=-1+2\frac{(\ell_{i}-1)}{m_{r}-1}.
\end{equation}
The dimensions of the reconstructed images are $n=2048^{2}$, that
is, $2048\times2048$ pixels.

\begin{figure}
\includegraphics[width=1\columnwidth]{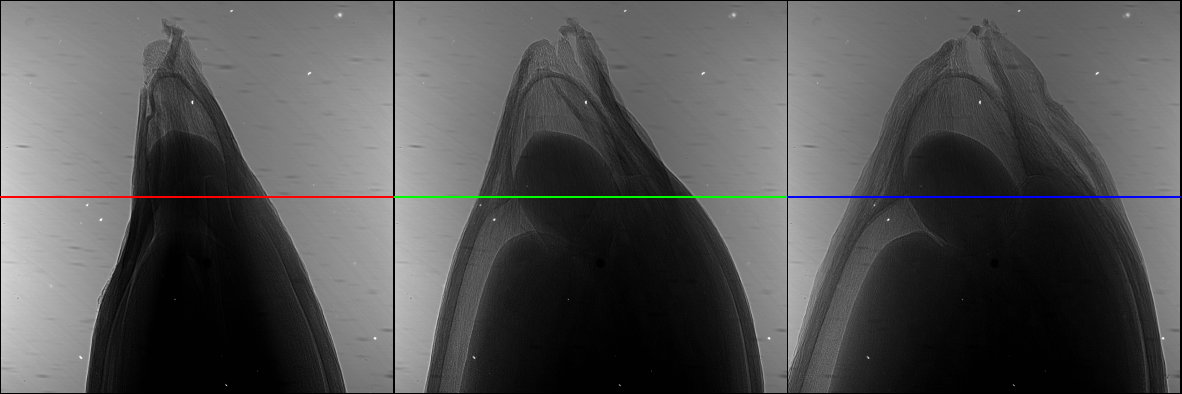}\vspace*{3pt}

\includegraphics[width=1\columnwidth]{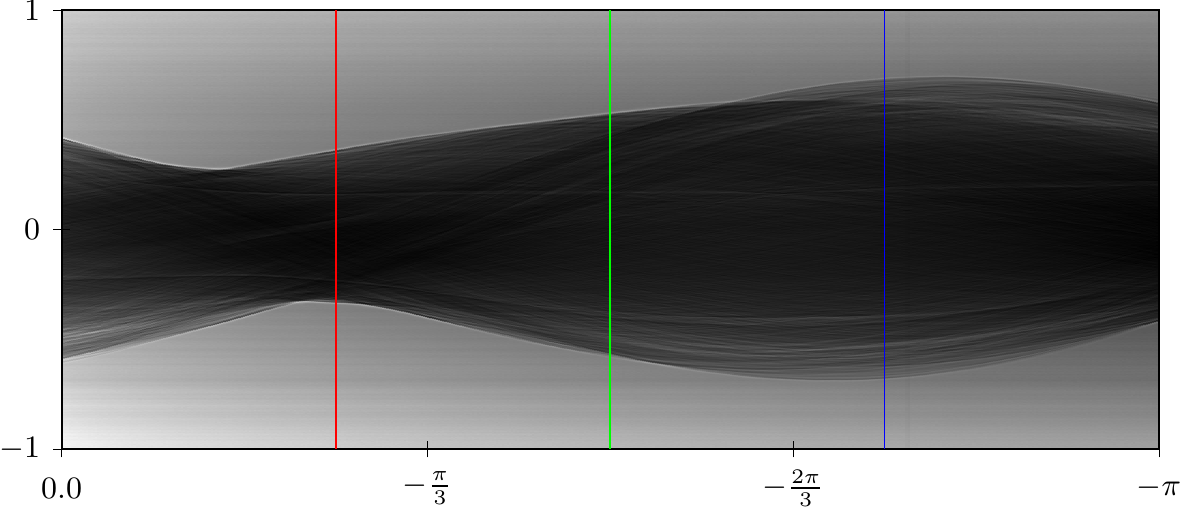}

\caption{Top: three of the $512$ images used in the reconstruction of a slice
of an apple seed. These images depict the number of photons recorded
on each sensor pixel, where brighter indicates more photons recorded.
The solid lines indicate data corresponding to the slice that will
be reconstructed in the experiments. Bottom: photon count data that
will be used for the reconstruction of the slice. Each of the $512$
rows of the data related to the same slice is used as a column in
a representation of the data as samples of the photon counts in the
$\theta\times t$ plane. The color of each of the solid lines identifies
the position where the corresponding data in the top row was used.\label{fig:fin_data}}
\end{figure}

In the transmission tomography setup found at the Brazilian Synchrotron
Light Laboratory, the object to be imaged is positioned between the
synchrotron x-ray source and a scintillator that is digitally imaged
by a CCD camera. Data from the $\kappa$-th readout from this camera
provides a $2048\times2048$ radio-graphic image that includes the
photon counts at pixels of the camera corresponding to integration
lines with parameters $\theta=-\pi(\kappa-1)/\left(m_{v}-1\right)$
and $t\in\left\{ -1+2\frac{(\ell-1)}{m_{r}-1}:\ell\in\left\{ 1,2,\dots,m_{r}\right\} \right\} $
not only for one two-dimensional slice, but simultaneously for $2048$
parallel slices of the imaged volume. Reconstruction of each of the
slices is later performed independently through the two-dimensional
model. Figure~\ref{fig:fin_data} depicts the procedure used to obtain
the photon count data $p_{i}$ that will be used for one slice reconstruction.

\begin{algorithm}
\caption{OISTA}
\label{algo:OISTA}

\begin{algorithmic}[1]

\STATE{\textbf{input }$L_{0}>0$, $\beta>1$, $\mathbf{x}_{0}\in\mathbb{R}^{n}$,
$N\in\mathbb{N}$}

\STATE{\textbf{set }$t_{1}=1$, $\mathbf{y}_{1}=\mathbf{x}_{0}$}

\STATE{\textbf{for }$k=1,...,N$}

\STATE{~~~~\textbf{set }$L_{k}=L_{k-1}$}

\STATE{~~~~\textbf{while }$\Psi\left(P_{L_{k}}(\mathbf{y}_{k})\right)>Q_{L_{k}}\left(P_{L_{k}}(\mathbf{y}_{k}),\mathbf{y}_{k}\right)$}

\STATE{~~~~~~~~\textbf{set $L_{k}=\beta L_{k}$}}

\STATE{~~~~\textbf{set }$\mathbf{x}_{k}=P_{L_{k}}(\mathbf{y}_{k})$}

\STATE{~~~~\textbf{set $t_{k+1}=\frac{1+\sqrt{1+4t_{k}^{2}}}{2}$}} 

\STATE{~~~~\textbf{set $\mathbf{y}_{k+1}=\mathbf{x}_{k}+\frac{t_{k}-1}{t_{k+1}}\left(\mathbf{x}_{k}-\mathbf{x}_{k-1}\right)+\frac{t_{k}}{t_{k+1}}\left(\mathbf{x}_{k}-\mathbf{y}_{k}\right)$}} 

\end{algorithmic}
\end{algorithm}

\paragraph*{Algorithms}

~We have compared several algorithms regarding speed of reduction
of objective function value. The methods experimented with are the
following: FISTA (Algorithm~\ref{algo:FISTA}) and OISTA (Algorithm~\ref{algo:OISTA}),
which we compare against FPGM (Algorithm~\ref{algo:FPGM}). In every
case, the starting image $\mathbf{x}_{0}$ was set to an image such
that all of its pixels have the same value and $\sum_{i=1}^{m}(R\mathbf{x}_{0})_{i}=\sum_{i=1}^{m}\tilde{b}_{i}$,
where $\tilde{b}_{i}$ are the estimates of the Radon Transform given
by $\tilde{b}_{i}=\ln\frac{\omega_{i}-d_{i}}{p_{i}-d_{i}}$. The parameter
$\beta$ was set to $2$ for all algorithms. This seemed a reasonable
choice and we did not feel the need to experiment with this parameter.
A first comparison uses $K=10$ and $\overline{\eta}=\infty$ for
FPGM and compares it to FISTA and OISTA. A second experiment considers
only FPGM, but with three different set of parameters $K$ and $\overline{\eta}$,
namely $(K,\overline{\eta})\in\{(10,\infty),(10,2),(\infty,\infty)\}$.
We denote the FPGM variations as FPGM$_{(K,\overline{\eta})}$ with
$(K,\overline{\eta})\in\{(10,\infty),(10,2),(\infty,\infty)\}$. The
reason why the value $\overline{\eta}=2$ was singled out among all
the possible finite values for this parameter is that when $\eta_{k}=2$
the FPGM iteration is identical to the one of OISTA, for which no
convergence proof exists. This way we show that our theoretical results
can explain why OISTA converges in some cases, while at the same time
they give insights on when and why OISTA might not converge. Notice
that the theoretical results presented above do not guarantee convergence
of FPGM$_{(\infty,\infty)}$ because it may not satisfy~\eqref{eq:decreasing_steps}.
In fact, Figure~\ref{fig:fpgm_params} shows what seems to be an
example of the algorithm not converging to an optimizer.

\paragraph*{Computational Issues}

~Because of the large data and image size, the matrix-vector products
$R\mathbf{x}$ and $R^{T}\mathbf{b}$ were implemented in a parallel
fashion and the computations were carried out by a Graphics Processing
Unit (GPU) using C++ CUDA in order to obtain low running times for
the reconstruction and thus to enable experimentation with a wide
range of parameters and for a very large number of iterations. A Python
wrapper for the GPU routines was used in order to implement all the
other algorithm parts using the NumPy library. We have used the ray-tracing
method from~\cite{hly99} for the on-the-fly computation of these
matrix-vector products. All other computations were accomplished by
the CPU. The computer had a GeForce Titan Xp GPU with 12GB of dedicated
RAM and a Ryzen 5 2600 CPU with 16GB of available RAM. The algorithms
are also practical when executed in less powerful hardware, as Figure~\ref{fig:convergence_plots-time-cpu}
shows. In practice far less than $1000$ iterations will be executed
and the algorithm will be executed with a single set of parameters,
which means that even modest CPUs such as the one used suffice for
routine reconstruction. The single precision GPU implementation runs
at around $12$ times faster than the double precision CPU implementation.

\begin{figure}
\includegraphics[width=1\columnwidth]{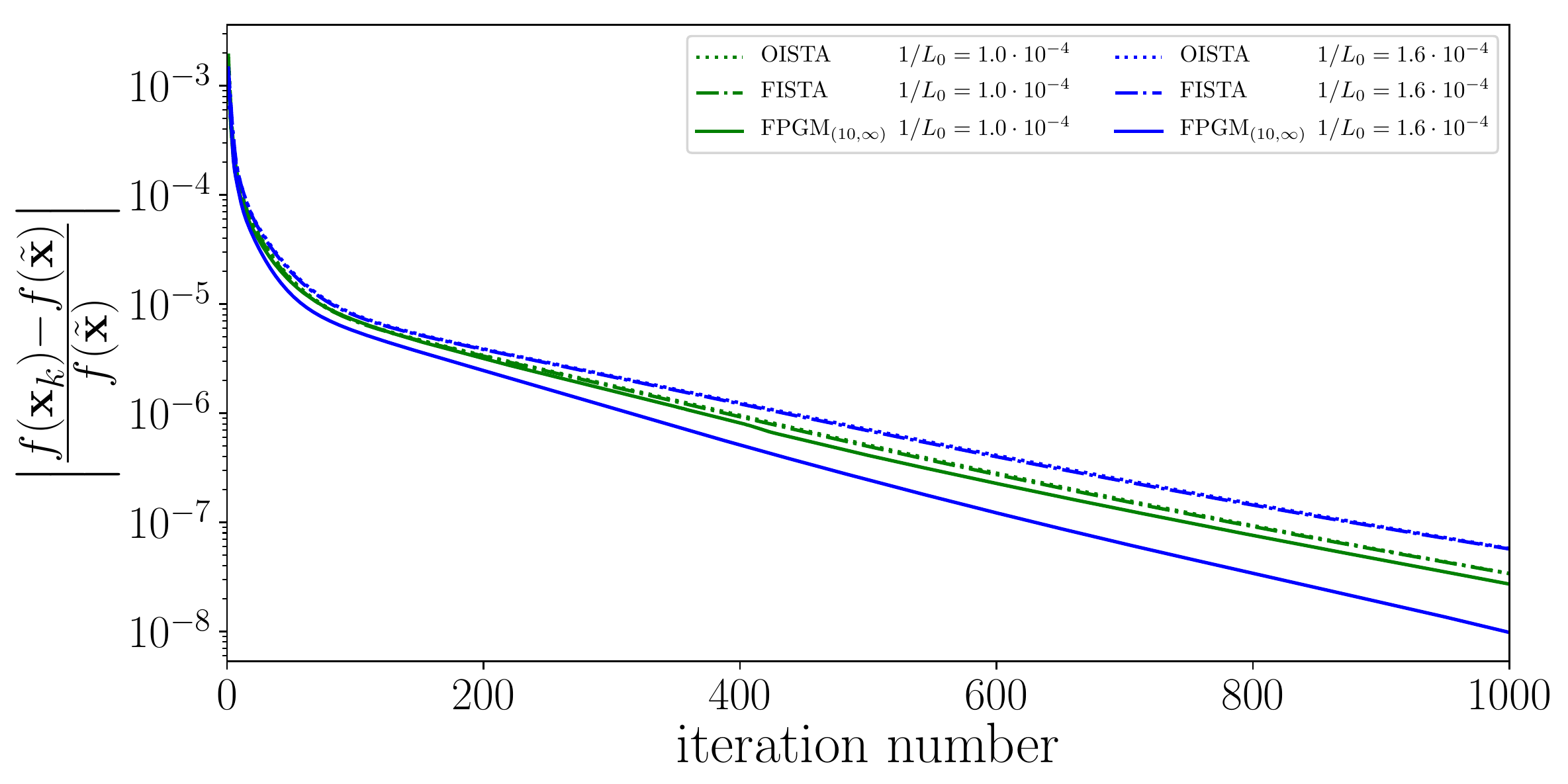}

\caption{Plot of objective function value at each iteration for two different
starting step-size values in three different algorithms.\label{fig:convergence_plots}}
\end{figure}

\begin{figure}
\includegraphics[width=1\columnwidth]{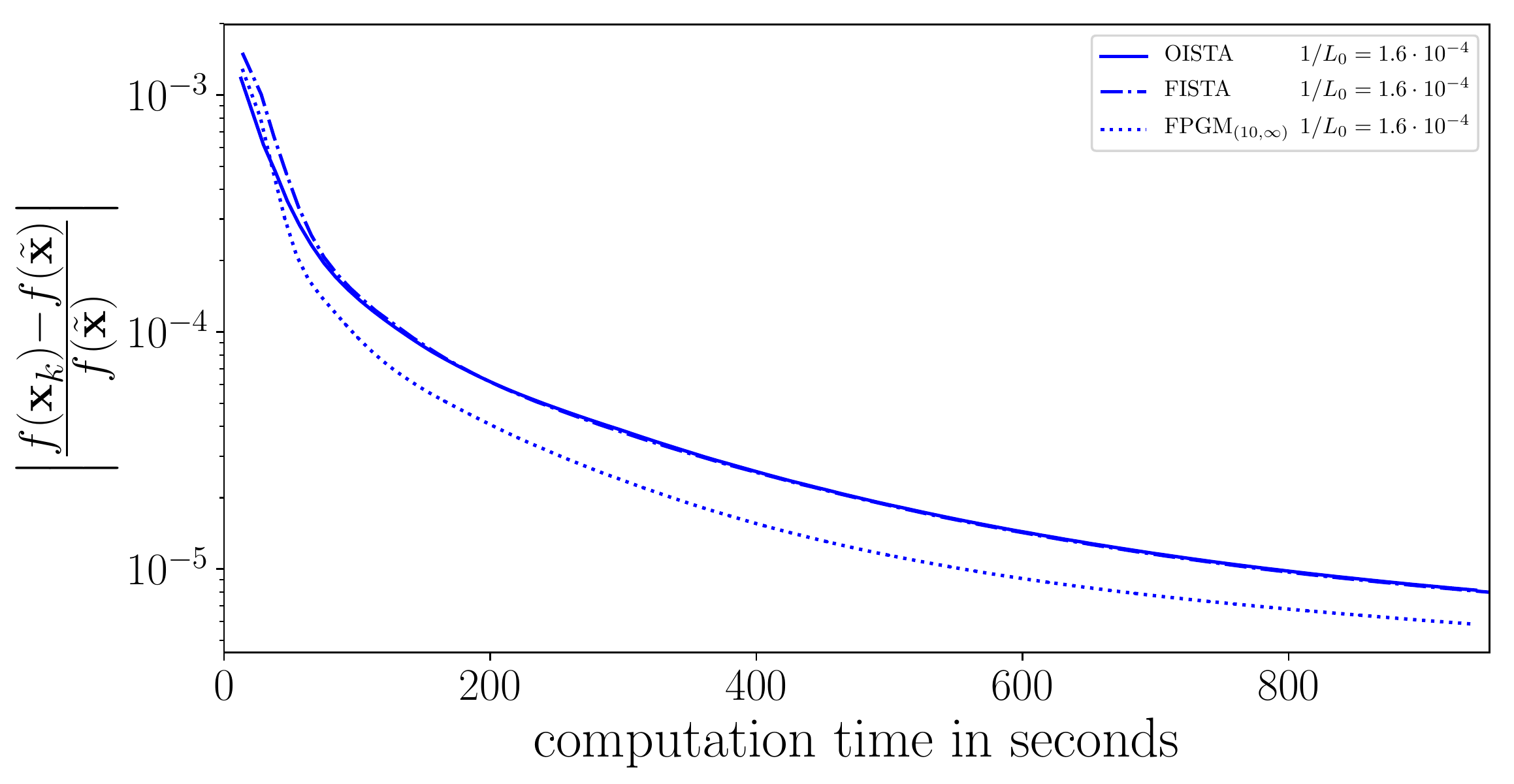}

\caption{Plot of objective function over time in three different algorithms,
running on the CPU.\label{fig:convergence_plots-time-cpu}}
\end{figure}

\paragraph*{Discussion}

~We base our discussion on log-plots of the objective function reduction
as iterations proceed. In these plots we represent the evolution of
the relative difference between the object function value obtained
by each algorithm and the value obtained by FISTA after $2000$ iterations
(using $\nicefrac{1}{L_{0}}=1.3\cdot10^{-4}$). We denote the reference
image obtained by this long run of FISTA by $\tilde{\mathbf{x}}$.
The starting stepsize for the long FISTA run was selected in order
to obtain good convergence and running FISTA for twice as many iterations
as the other methods ensured that the final iterations obtained this
way by FISTA had better objective function value than the other methods.

For a comparison of FPGM with existing algorithms, we have tested
the behavior of FPGM$_{(10,\infty)}$ and compared it to FISTA and
OISTA. We notice in Figure~\ref{fig:convergence_plots} that FPGM$_{(10,\infty)}$
is very competitive regardless of the starting stepsize adopted within
the experimented range. We have tried running the algorithm for a
variety of stepsizes within the range $\nicefrac{1}{L_{0}}\in[9\cdot10^{-5},2\cdot10^{-4}]$.
This range was chosen because it contains the value of $\nicefrac{1}{L_{1000}}$
obtained by the line search procedure if values above this range were
used. This means that smaller stepsize values unnecessarily slow down
the algorithms and should not be used, whereas for larger values the
line search reduces the stepsize in the first few iterations to something
within the adopted range. For the sake of legibility, we present the
results obtained for only two values of $\nicefrac{1}{L_{0}}$ showing
the typical behavior. We note that, for a favorable selection of starting
stepsize FPGM outperforms noticeably the other methods, whereas for
some values of the starting stepsize the improvements may be small.

\begin{figure}
\includegraphics[width=1\columnwidth]{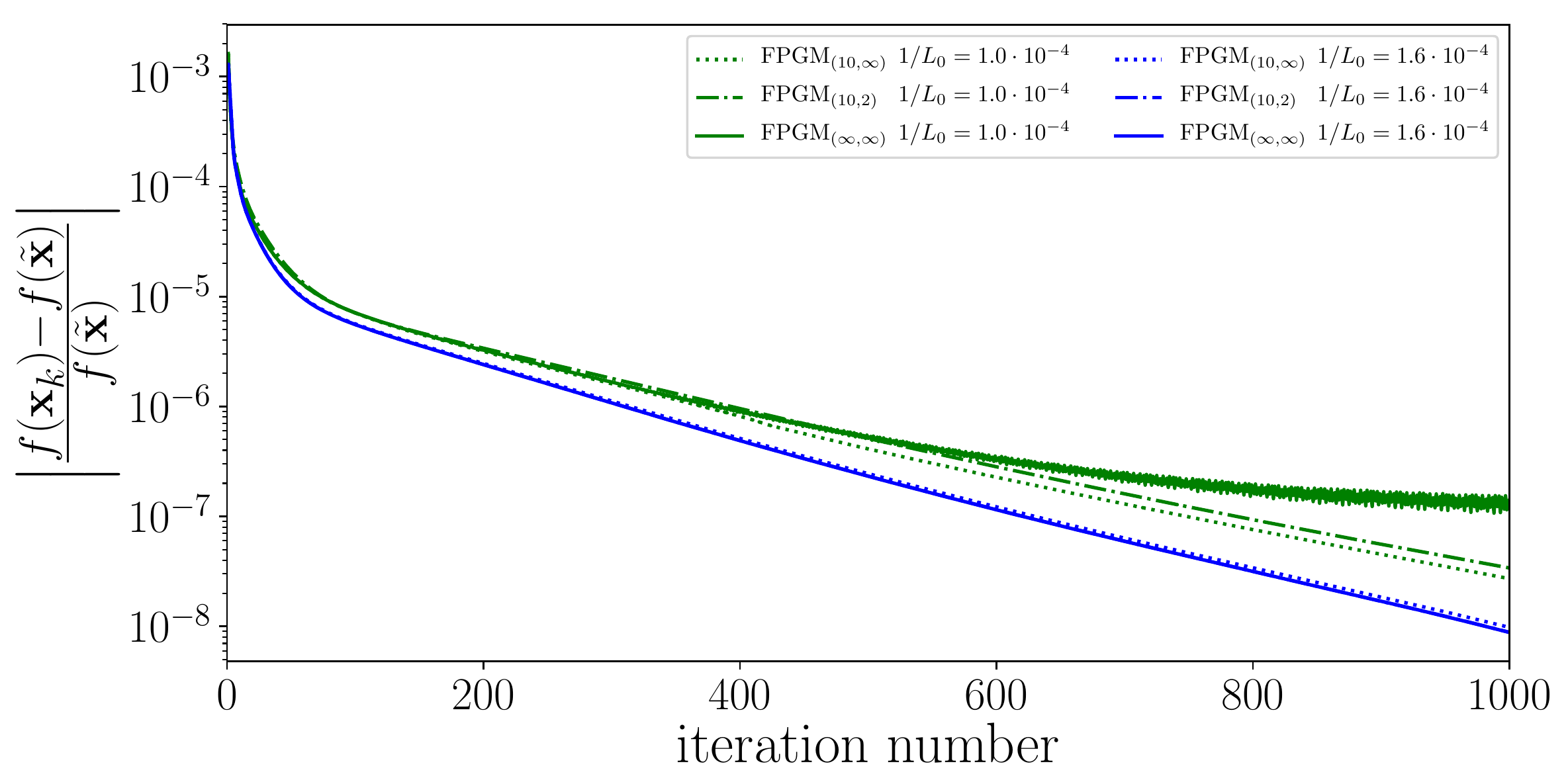}

\caption{Comparison of the behavior of the objective function value of the
iterates of FPGM$_{(K,\overline{\eta})}$ for some values of parameters
$K$, $\overline{\eta}$, and $L_{0}$. Notice that because FPGM$_{(\infty,\infty)}$
might not converge, oscillatory behavior is observed for some starting
stepsize values.\label{fig:fpgm_params}}
\end{figure}

Figure~\ref{fig:fpgm_params} illustrates an experiment designed
to investigate sensitivity of Algorithm~\ref{algo:FPGM} to the choice
of parameters $K$ and $\overline{\eta}$. As it can be seen, at least
for this combination of optimization model and data, FPGM$_{(10,\infty)}$
and FPGM$_{(10,2)}$, which are provably convergent, appear to have
advantages over FPGM$_{(\infty,\infty)}$, which does not necessarily
satisfy~\eqref{eq:decreasing_steps} and, therefore, has no proven
convergence. Also, it seems that limiting the value of $\eta_{k}$
does not necessarily reduce convergence speed when all other parameters
are the same. Convergent versions of FPGM did not present any drastic
performance changes within the set of parameters used, which is a
sign of robustness with relation to parameter selection.

\subsection{Reconstruction from Simulated Data}

In the present subsection we perform reconstruction from simulated
data in order to be able to measure reconstruction quality. We have
performed several similar experiments with the intention of performing
statistical testing on the hypothesis that one algorithm performs
better than the other with respect to a well defined mathematical
figure of merit. We will first describe how each reconstruction experiment
was performed and then we discuss the results of the statistical hypothesis
testing.

\begin{figure}
\includegraphics[width=0.5\columnwidth]{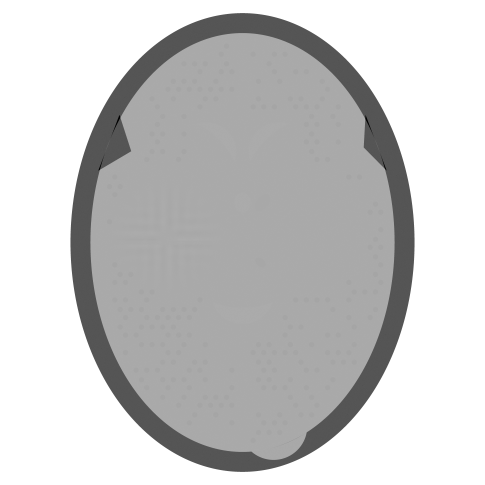}\includegraphics[width=0.5\columnwidth]{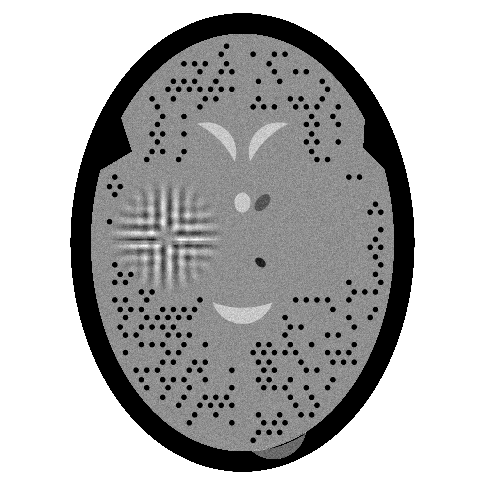}\caption{One of the possible phantoms used in the statistical hypothesis testing
experiment. On the left, the full range of attenuation values of the
phantom is shown, its minimum value represented as white and the maximum
as black. On the right, values below $0.204$~cm$^{-1}$ are shown
as white and values above $0.21765$~cm$^{-1}$ are shown as black.\label{fig:phantom}}
\end{figure}

\begin{figure}
\setlength{\fboxsep}{0pt}%
\setlength{\fboxrule}{0.5pt}%
\fbox{\includegraphics[width=0.33333\columnwidth]{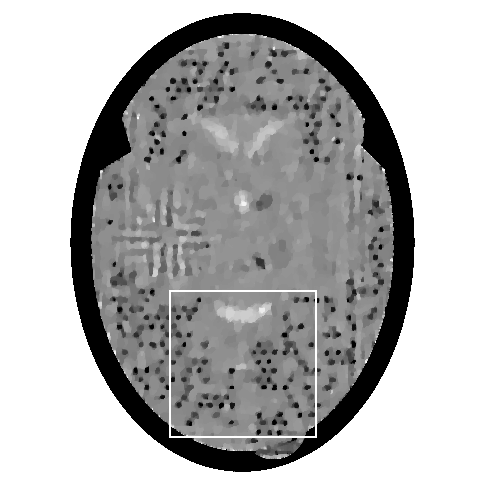}}\hspace*{-0.5pt}%
\fbox{\includegraphics[width=0.33333\columnwidth]{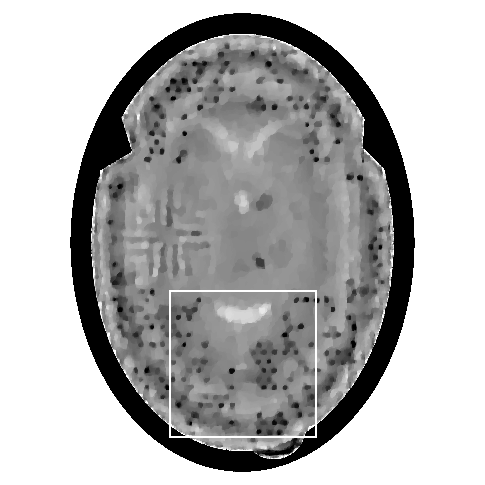}}\hspace*{-0.5pt}%
\fbox{\includegraphics[width=0.33333\columnwidth]{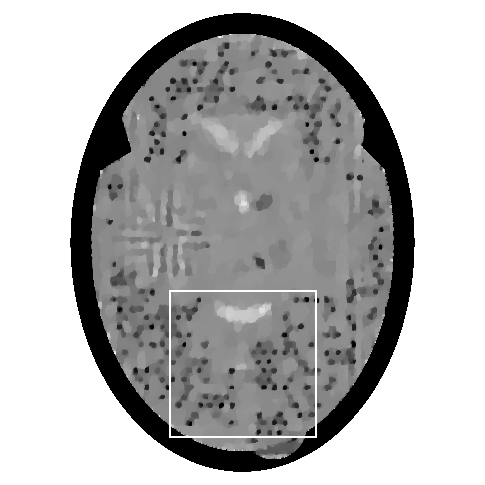}}\\[-3.5pt]
\fbox{\includegraphics[width=0.33333\columnwidth]{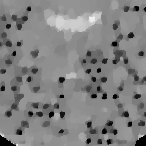}}\hspace*{-0.5pt}%
\fbox{\includegraphics[width=0.33333\columnwidth]{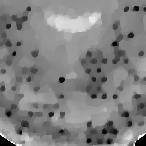}}\hspace*{-0.5pt}%
\fbox{\includegraphics[width=0.33333\columnwidth]{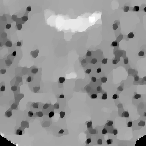}}

\caption{Left: First iteration of SupART to reach stopping criterion. Center:
$50^{\text{th}}$ iteration of FPGM$_{(10,\infty)}^{5\cdot10^{-3}}$.
Right: $100^{\text{th}}$ iteration of FPGM$_{(10,\infty)}^{5\cdot10^{-3}}$.
Top: full images. Bottom: details for easier inspection.}
\label{fig:reconstructions}
\end{figure}

\paragraph*{Test Images and Projection Data}

~Each reconstruction experiment was based on a mathematical phantom
defined from simple geometrical features, which can be seen in Figure~\ref{fig:phantom}.
It has a skull-like high-attenuation feature around it, while possessing
less contrasting objects inside the skull. Among the features inside
the skull there are small similar circular objects scattered inside
the inner structure. The number of such objects is constant, but its
position vary on each experiment in the following form: each object
must exist in a position or else to exist in the same position except
that mirrored along the central vertical axis. The side where the
circular object will be placed is selected at random with equal probability
of selection for left and right placement. The other features of the
phantom do not change from one experiment to another, except for a
small random inhomogeneity, which may also be different from one experiment
to another. Generation of the images, simulation of tomographic data
from each generated image, and reconstruction from each of the simulated
dataset was automatically handled by the \textquotedbl{}experimenter\textquotedbl{}
feature of SNARK14 software.\footnote{SNARK14 may be downloaded free of charge from http://turing.iimas.unam.mx/SNARK14M/}
For this set of experiments, the sequential implementations provided
by SNARK14 were used instead of the specially crafted parallel algorithms
used in the previous set of experiments.

For these experiments, data acquisition followed a divergent beam
geometry where detectors are distributed in a circular strip centered
at the X-ray source with source-to-origin distance (that is, the radius
of the source's circular path followed during acquisition) equal to
$78$ cm and with source-to-detector distance (that is, the radius
of the circular strip of detectors) equal to $110.735$ cm. The simulated
circular detector strip contained $693$ detectors with the spacing
between the center of successive detectors equal to $0.0533$ cm.
Each projection contains data from all $693$ detectors and $180$
of such projections were measured, $2^{\circ}$ apart from each other.
Reconstructed images had $485\times485$ pixels covering the square
$[-9.1\text{ cm},9.1\text{ cm}]^{2}$.

\paragraph*{Reconstruction Model}

~In these experiments, for the reconstruction with FPGM we used the
least squares model with Total Variation regularization. Let us first
define the Total Variation as
\begin{equation}
TV(\mathbf{x}):=\sum_{i=1}^{m}\sqrt{\left(x_{i}-x_{r(i)}\right)^{2}+\left(x_{i}-x_{a(i)}\right)^{2}},\label{eq:TV}
\end{equation}
where $r(i)$ is the index of the pixel to the right of pixel $i$
and $a(i)$ is the index of the pixel above pixel~$i$. The Total
Variation regularization functional was chosen because of its computational
convenience and good regularization properties for tomographic images.

Our model therefore is
\begin{equation}
\min_{\mathbf{x}\in\mathbb{R}_{+}^{n}}\quad\left\Vert R\mathbf{x}-\mathbf{b}\right\Vert ^{2}+\lambda TV(\mathbf{x}),\label{eq:TV_model}
\end{equation}
where $R\in\mathbb{R}^{m\times n}$ is the matrix representing the
discretization of the Radon Transform and \textbf{$\mathbf{b}\in\mathbb{R}^{n}$}
contains the simulated data. This model can be rewritten as
\begin{equation}
\min_{\mathbf{x}\in\mathbb{R}^{n}}\quad f(\mathbf{x})+\phi(\mathbf{x}),
\end{equation}
where $f(\mathbf{x})=\left\Vert R\mathbf{x}-\mathbf{b}\right\Vert ^{2}$
and $\phi(\mathbf{x})=\lambda TV(\mathbf{x})+\chi_{\mathbb{R}_{+}^{n}}(\mathbf{x})$.
The proximal operator for $\phi(\mathbf{x})=\lambda TV(\mathbf{x})+\chi_{\mathbb{R}_{+}^{n}}(\mathbf{x})$
was approximately computed by $10$ steps of the fast gradient method
for the dual formulation of the constrained $TV$ proximal problem,
according to~\cite{Beck2009a}.

\paragraph*{ART}

~One of the tested algorithms is FPGM$_{(10,\infty)}$, as described
above. The other is a superiorized version of the Algebraic Reconstruction
Technique (ART). ART is a common denomination for a series of algorithms
that sequentially update the reconstructed image using one row of
the system $R\mathbf{x}=\mathbf{b}$ (where $\mathbf{b}$ is the tomographic
data and $R$ is the projection matrix) at a time. The meaning of
\textquotedbl{}superiorized\textquotedbl{} will be explained after
we describe ART itself. The precise version of ART used here is described
as Algorithm~\ref{algo:ART} below, where $\mathbf{r}_{i}$ is the
$i$-th row of the matrix $R$ (an $n$-dimensional row-vector).

\begin{algorithm}
\caption{ART}
\label{algo:ART}

\begin{algorithmic}[1]

\STATE{\textbf{input }$\alpha\in\mathbb{R}$, $\left\{ i_{0},i_{1},\dots,i_{m-1}\right\} \subset\mathbb{N}$,
$\mathbf{x}\in\mathbb{R}^{n}$, $\mathbf{b}\in\mathbb{R}^{m}$}

\STATE{\textbf{set }$\mathrm{\mathbf{y}_{0}}=\mathbf{x}$}

\STATE{\textbf{for }$\ell=0,...,m-1$}

\STATE{~~~~\textbf{set }$\mathbf{y}_{\ell+1}=\mathbf{y}_{\ell}-\alpha\frac{\left\langle \mathbf{r}_{i_{\ell}}^{T},\mathbf{y}_{\ell}\right\rangle -\mathbf{b}_{i_{\ell}}}{\left\Vert \mathbf{r}_{i_{\ell}}^{T}\right\Vert ^{2}}\mathbf{r}_{i_{\ell}}^{T}$}

\STATE{\textbf{return $\mathbf{y}_{m}$}} 

\end{algorithmic}
\end{algorithm}

We will denote the result of the application of Algorithm~\ref{algo:ART}
to a given $\mathbf{x}\in\mathbb{R}^{n}$ as ART$(\mathbf{x})$, omitting
the other input parameters that will be kept fixed as follows: $\alpha=0.05$,
which is known as the relaxation parameter, $\mathbf{b}\in\mathbb{R}^{m}$
which is the tomographic data, and the sequence $\left\{ i_{0},i_{1},\dots,i_{m-1}\right\} $,
that determines the order of processing of the tomographic data within
the full cycle of Algorithm~\ref{algo:ART}, which is defined according
to the \textquotedbl{}efficient\textquotedbl{} sequence described
in~\cite{hem93}. This sequence was shown experimentally to improve
convergence of the method.

\paragraph*{Superiorization}

~Superiorization is a technique for the modification of iterative
optimization (or feasibility) algorithms that possess resiliency to
summable perturbations. It consists of replacing the iteration $\mathbf{x}_{k}$
of a given algorithm by a perturbed version $\tilde{\mathbf{x}}_{k}$
prior to applying the computations that will lead to the next iteration
$\mathbf{x}_{k+1}$ of the method. The perturbations, which we may
denote as $\mathbf{s}_{k}:=\tilde{\mathbf{x}}_{k}-\mathbf{x}_{k}$,
are usually designed to improve the iterates according to a secondary
criterion represented by a function $\phi:\mathbb{R}^{n}\to\mathbb{R}$.
By an improvement we mean a reduction of the value of $\phi$, that
is, $\mathbf{s}_{k}$ is normally designed so that $\phi\left(\mathbf{x}_{k}+\mathbf{s}_{k}\right)\leq\phi\left(\mathbf{x}_{k}\right)$.
In our experiments we have used $\phi=TV$, i.e., the Total Variation
defined in~\eqref{eq:TV} is the secondary criterion for superiorization.

Algorithm~\ref{algo:SupTV} describes the procedure that we use to
superiorize ART in our experiments. It makes use of the concept of
non-ascending vectors. We define the set $\tilde{\partial}\phi(\mathbf{x})$
of non-ascending vectors for a convex function $\phi:\mathbb{R}^{n}\to\mathbb{R}$
from a point $\mathbf{x}\in\mathbb{R}^{n}$ as the set of vectors
$\mathbf{t}\in\mathbb{R}^{n}$ with $\|\mathbf{t}\|\leq1$ and such
that there is $B>0$ for which $b\in[0,B]$ implies that $\phi(\mathbf{x}+b\mathbf{t})\leq\phi(\mathbf{x})$.
Notice that $\mathbf{0}\in\tilde{\partial}\phi(\mathbf{x})$, meaning
that the set of non-ascending vectors is never empty. Following~\cite[Theorem\ 2]{hgd12},
we can construct a non-ascending vector $\mathbf{t}$ for $\phi$
from $\mathbf{x}$ as follows, first componentwise defining an auxiliary
vector $\tilde{\mathbf{t}}\in\mathbb{R}^{n}$ as
\begin{equation}
\tilde{t}_{i}=\begin{cases}
\frac{\partial\phi}{\partial x_{i}}(\mathbf{x}) & \text{if }\frac{\partial\phi}{\partial x_{i}}(\mathbf{x})\text{ is well defined}\\
0 & \text{otherwise,}
\end{cases}\quad\text{and then defining\quad\ensuremath{\mathbf{t}}}=\begin{cases}
-\frac{\tilde{\mathbf{t}}}{\|\tilde{\mathbf{t}}\|} & \text{if }\tilde{\mathbf{t}}\neq\mathbf{0}\\
\mathbf{0} & \text{otherwise.}
\end{cases}
\end{equation}
We will denote the output of Algorithm~\ref{algo:SupTV} for given
$\mathbf{x}\in\mathbb{R}^{n}$ and $l\in\mathbb{N}$ as SupTV$(\mathbf{x},l)$,
omitting parameters $a$, \textbf{$b$} and $I$, which were kept
fixed with the values $a=1-10^{-4}$, $b=3\cdot10^{-2}$, and $I=10$.

\begin{algorithm}
\caption{SupTV}
\label{algo:SupTV}

\begin{algorithmic}[1]

\STATE{\textbf{input} $\mathbf{x}\in\mathbb{R}^{n}$, $l\in\mathbb{N}$,
$a\in(0,1)$, $b\in(0,\infty),$ $I\in\mathbb{N}$}

\STATE{\textbf{set} $\mathbf{y}=\mathbf{x}$}

\STATE{\textbf{for $i=1,2,\dots,I$}}

\STATE{$\quad$\textbf{set }$\mathbf{t}\in\tilde{\partial}TV(\mathbf{y})$}

\STATE{$\quad$\textbf{do}}

\STATE{$\quad$$\quad$\textbf{set }$\tilde{b}=ba^{l}$}

\STATE{$\quad$$\quad$\textbf{set }$\tilde{\mathbf{y}}=\mathbf{y}+\tilde{b}\mathbf{t}$}

\STATE{$\quad$$\quad$\textbf{set }$l=l+1$}

\STATE{$\quad$\textbf{while} $TV(\tilde{\mathbf{y}})>TV(\mathbf{y})$}

\STATE{$\quad$\textbf{set }$\mathbf{y}=\tilde{\mathbf{y}}$}

\STATE{$\textbf{return}$ $(\mathbf{y},l)$}

\end{algorithmic}
\end{algorithm}

It is important to notice that superiorization according to the criterion
$\phi$ does not mean minimization of that criterion. This makes the
approach much more flexible than optimization because the theory is
easier to develop for a large variety of secondary criteria. On the
other hand, it is clear that, because there is no guarantee of minimization,
there may exist images that are equally data-consistent than those
obtained by superiorization, but that are better according to $\phi$.
In practice, however, superiorized methods have been shown to be almost
as fast as the non-superiorized version of the same method, to be
very flexible and simple to use in practice, and to deliver considerably
improved results when compared to its non-superiorized counterpart.

\begin{algorithm}
\caption{SupART}
\label{algo:SupART}

\begin{algorithmic}[1]

\STATE{\textbf{input} $\mathbf{x}_{0}\in\mathbb{R}^{n}$, $\epsilon\in\mathbb{R}$,
$\mathbf{b}\in\mathbb{R}^{m}$}

\STATE{\textbf{set} $l_{0}=0$, $k=0$}

\STATE{\textbf{while $\left\Vert R\mathbf{x}_{k}-\mathbf{b}\right\Vert ^{2}>\epsilon$}}

\STATE{$\quad$\textbf{set }$\left(\mathbf{x}_{k+\nicefrac{1}{2}},l_{k+1}\right)=\text{SupTV}\left(\mathbb{\mathbf{x}}_{k},l_{k}\right)$}

\STATE{$\quad$\textbf{set }$\mathbf{x}_{k+1}=\text{ART}\left(\mathbb{\mathbf{x}}_{k+\nicefrac{1}{2}}\right)$}

\STATE{$\quad$\textbf{set }$k=k+1$}

\STATE{$\textbf{return}$ $\mathbf{x}_{k}$}

\end{algorithmic}
\end{algorithm}

\paragraph*{SupART}

~We can finally precisely define the algorithm that we have compared
to FPGM in the experiments described in the present subsection. Each
iteration of Superiorized ART (SupART), presented in Algorithm~\ref{algo:SupART},
is defined as application of Algorithm~\ref{algo:SupTV} followed
by application of Algorithm~\ref{algo:ART}. Unlike FPGM which uses
a prescribed number of iterations, our version of SupART uses a stopping
criterion based in data-fidelity, as measured by the squared residual
$\|R\mathbf{x}-\mathbf{b}\|^{2}$. We have done so in order to force
the images obtained by SupART to have a data-fidelity value that is
similar to those obtained by FPGM. We will describe precisely how
it was done below, when discussing the parameter selection for the
algorithms and models.

\paragraph*{Parameter Selection}

~Algorithm FPGM$_{(10,\infty)}$ was used with parameters $L_{0}=1$,
$\beta=2$, and $N=100$ in order to solve model~\eqref{eq:TV_model}.
In this case we are not comparing convergence speed between variations
of the same algorithm, so the choice of parameters is not critical
and these values work well in practice. For the model to be fully
defined, the value of regularization parameter $\lambda$ has to be
specified. We now describe how we have decided which values to use
for $\lambda$: We have generated a phantom and simulated tomographic
data from this phantom as described above, and then ran the method
for the set of parameters $\lambda\in\{10^{-4},5\cdot10^{-4},10^{-3},5\cdot10^{-3},10^{-2},5\cdot10^{-2}\}$.
We verified that, among these, the value $\lambda=5\cdot10^{-3}$
provided the highest value of the Imagewise Region of Interest (IROI)
figure of merit (the IROI is precisely described below, see~\eqref{eq:IROI}).
Then, we have selected two other values for the regularization parameter
within the same order of magnitude, leading to the final set of regularization
parameters that were used in the final algorithmic comparison: $\lambda\in\{4,5,6\}\cdot10^{-3}$.

SupART has already been tested for this reconstruction problem in~\cite{gah17},
and we have used the (already mentioned above) same set of parameters
as in that reference ($\alpha=0.05$, $a=1-10^{-4}$, $b=3\cdot10^{-2}$)
with one exception. In~\cite{gah17} the authors have used $I=40$,
while we have used $I=10$ in order to match the number of the iterations
used by the approximation algorithm for the proximal operator of the
$TV$ that is used in FPGM. SupART also requires the determination
of the stopping parameter $\epsilon$, which we have selected to be
$2.33$ because this is, up to three digits accuracy, what FPGM$_{(10,\infty)}$
obtained after $100$ iterations when solving~\eqref{eq:TV_model}
with $\lambda=5\cdot10^{-3}$.

Finally, both FPGM and SupART require the selection of a starting
image $\mathbf{x}_{0}$, which we have selected to be a uniform image
such that $\sum_{i=1}^{m}\left(R\mathbf{x}_{0}\right)_{i}=\sum_{i=1}^{m}b_{i}$.

\paragraph*{Image Reconstruction Evaluation}

~Comparison of the tested algorithms was done according to the IROI
figure of merit. The IROI is defined in~\cite{nah99} and we provide
a succinct explanation of it here. The IROI considers only paired
structures, which in our case were each of the small circles (tumors)
in the phantom shown at Figure~\ref{fig:phantom} and its empty counterpart
on the other side of the phantom. We associate each of these pairs
of structures with an index ranging from $1$ to $S$. Let, for $s\in\{1,2,\dots,S\}$,
$\alpha_{t}^{p}(s)$ be the average density in the phantom on the
region where the tumor $s$ is present, and $\alpha_{n}^{p}(s)$ be
the average density in the phantom on the respective region where
the tumor $s$ is not present. Similarly, $\alpha_{t}^{r}(s)$ and
$\alpha_{n}^{r}(s)$ are the average values of these regions in the
reconstructed images. Then, the value of the IROI is given by
\begin{equation}
\frac{\sum_{s=1}^{S}\bigl(\alpha_{t}^{r}(s)-\alpha_{n}^{r}(s)\bigr)}{\sum_{s=1}^{S}\left(\alpha_{n}^{r}(s)-\frac{1}{S}\sum_{s'=1}^{S}\alpha_{n}^{r}(s')\right)^{2}}/\frac{\sum_{s=1}^{S}\bigl(\alpha_{t}^{p}(s)-\alpha_{n}^{p}(s)\bigr)}{\sum_{s=1}^{S}\left(\alpha_{n}^{p}(s)-\frac{1}{S}\sum_{s'=1}^{S}\alpha_{n}^{p}(s')\right)^{2}}.\label{eq:IROI}
\end{equation}
Large IROI values mean that we have large differences $\alpha_{t}^{r}(s)-\alpha_{n}^{r}(s)$
and, therefore, that reconstructed images with large IROI values tend
to offer easier detectability of the tumors in the regions of interest
than images with low IROI.

The last sentence indicates the essential nature of IROI; its purpose
is to measure the efficacy of an algorithm for tumor detectability
in the reconstructions. It is different from measures based on the
overall similarity between the phantoms and their reconstructions,
such as structural similarity (SSIM). Those other measures are likely
to indicate better the perceived quality of the whole reconstruction,
but IROI is more useful for evaluating the diagnostic efficacy of
a reconstruction algorithm for the specific task of tumor detection
\cite{nah99}.

\paragraph*{Discussion}

~We have performed $30$ repetitions of the cycle of \textbf{(i)}
generating a random phantom, \textbf{(ii)} simulating projection data
from the phantom, \textbf{(iii)} reconstructing images from the simulated
data with SupART and FPGM (for the three versions of model~\eqref{eq:TV_model}
with $\lambda\in\{4,5,6\}\cdot10^{-3}$), and \textbf{(iv)} computing
the IROI of the reconstructed images. These $30$ repetitions were
used for statistical testing of the null hypothesis (H$_{0}$) \textquotedbl{}SupART
reconstructs an image with IROI equally large as the one obtained
by FPGM$_{(10,\infty)}^{\lambda}$\textquotedbl{}, where FPGM$_{(10,\infty)}^{\lambda}$
denotes FPGM$_{(10,\infty)}$ applied to the solution of~\eqref{eq:TV_model}
with regularization parameter $\lambda$. The results can be seen
on Table~\ref{tab:results_experimenter}, where we have denoted the
estimated likelihood of rejecting H$_{0}$ while it is true by $P_{0}$.

Figure~\ref{fig:reconstructions} shows reconstructions obtained
by SupART and FPGM$_{(10,\infty)}^{5\cdot10^{-3}}$ in the last of
the $30$ experiments for an illustration of the difference of the
results of the algorithms. We remark that although FPGM is a fast
algorithm, certain features take a large number of iterations to be
accurately captured by the method. Notice how a low-frequency low-contrast
artifact in the form of radial waves can be seen in the $50^{\text{th}}$
iteration of FPGM$_{(10,\infty)}$. This artifact is not present if
the algorithm is allowed to run for more time. This is not a feature
exclusively present in FPGM nor it is exclusive to the particular
selection of algorithmic parameters that were used in our experiments,
but rather a common issue with all the fast proximal gradient methods
which might not have been noticed before because low-contrast features
such as those in the phantom we use are not usually investigated in
simulated reconstructions. Indeed, even Conjugate Gradient methods
seem to present such kind of artifact (see, e.g., Figure~4(b) in~\cite{zch18}).

Although it can be seen that in this particular setting the FPGM reconstructions
are consistently better than SupART reconstructions, we make no claims
that FPGM performs better than SupART in general because the experiments
were performed for a particular set of simulation and algorithmic
parameters. It was not our intent to optimize SupART's parameter for
this specific reconstruction problem, but it may well be the case
that SupART could obtain results as good (or better) as those obtained
by FPGM if the parameters were more carefully adjusted for the particular
problem at hand. However, the experiments allow us to conclude that
FPGM is competitive with state-of-the-art methods in terms of image
quality reconstruction.

\begin{table}
\begin{centering}
\begin{tabular}{ccc}
Algorithm & Mean IROI & $P_{0}$\tabularnewline
\hline 
FPGM$_{(10,\infty)}^{4\cdot10^{-3}}$ & $0.1914$ & $3.093\cdot10^{-8}$\tabularnewline
FPGM$_{(10,\infty)}^{5\cdot10^{-3}}$ & $0.1883$ & $3.437\cdot10^{-8}$\tabularnewline
FPGM$_{(10,\infty)}^{6\cdot10^{-3}}$ & $0.1623$ & $1.342\cdot10^{-7}$\tabularnewline
SupART & $0.1438$ & N/A\tabularnewline
\hline 
\end{tabular}\medskip{}

\par\end{centering}

\centering{}\caption{Mean IROIs (larger is better) and estimated probabilities of rejecting
a true null hypothesis.}
\label{tab:results_experimenter}
\end{table}

\section{Conclusions and Future Work\label{sec:Conclusions-and-Future}}

We have introduced a new algorithm called FPGM with convergence rate
$O(1/k^{2})$ for the minimization of a separable smooth plus nonsmooth
convex function. The convergence theory we have developed for this
algorithm explains why OISTA~\cite{Kim2016c} converges in many circumstances
even though in the worst-case scenario OISTA is non-convergent~\cite{thg17}.
In large-scale experiments with real world data, FPGM consistently
outperforms, from the viewpoint of objective function reduction, FISTA,
and in many cases FPGM performs better than OISTA in practice. Experiments
performed from simulated data showed that FPGM presents reconstructed
images of good quality, but it was noticed that a large number of
iterations may be required in order to eliminate low-contrast artifacts
that are present in the early iterations of the algorithm. In future
research on this subject we plan to investigate the causes of such
artifacts and to devise ways of avoiding them occurring in the early
iterates~\cite{Aharon2006}.

\section*{Acknowledgments}

This study has support by NIH grants R01-AR060238, R01-AR067156, and
R01-AR068966, and was performed under the rubric of the Center of
Advanced Imaging Innovation and Research (CAI$^{2}$R), a NIBIB Biomedical
Technology Resource Center (NIH P41-EB017183). E. S. Helou was partially
supported by FAPESP grants 2013/07375-0 and 2016/24286-9, and CNPq
grant 310893/2019-4. We are also indebted to LNLS for the beam time
granted through proposal number 20160215 and to NVIDIA for the GPU
used in the experiments, which was obtained through the NVIDIA Higher
Education and Research Grants. 

\bibliographystyle{spbasic}
\bibliography{refs_EH_2020_04_13}

\end{document}